\documentclass[1 [leqno,11pt]{amsart}
\usepackage{amssymb, amsmath,latexsym,amsfonts,amsbsy, amsthm,mathtools,graphicx,CJKutf8,CJKnumb,CJKulem,color}
\usepackage{float}
\usepackage{hyperref}

\numberwithin{equation}{section}

\allowdisplaybreaks

\let\al=\alpha
\let\b=\beta

\let\d=\delta
\let\e=\varepsilon

\let\f=\frac

\let\om=\omega

\let\na=\nabla

\let\pa=\partial

\def\cE{{\mathcal E}}
\def\cF{{\mathcal F}}
\def\cG{{\cal G}}
\def\cH{{\mathcal H}}

\def\cL{{\mathcal L}}
\def\cM{{\mathcal M}}

\def\cG{{\mathcal G}}

\def\R{\mathbb R}

\def\C{\mathbb C}
\def\N{\mathbb N}
\def\T{\mathbb T}

\newcommand{\beq}{\begin{equation}}
\newcommand{\eeq}{\end{equation}}
\newcommand{\ben}{\begin{eqnarray}}
\newcommand{\een}{\end{eqnarray}}
\newcommand{\beno}{\begin{eqnarray*}}
\newcommand{\eeno}{\end{eqnarray*}}

\newtheorem{theorem}{Theorem}[section]

\newtheorem{lemma}[theorem]{Lemma}
\newtheorem{proposition}[theorem]{Proposition}

\begin{document}

\title[The Oseen vortices operator]
{Pseudospectral and spectral bounds for the Oseen vortices operator}

\author{Te Li}
\address{School of Mathematical Science, Peking University, 100871, Beijing, P. R. China}
\email{little329@163.com}

\author{Dongyi Wei}
\address{School of Mathematical Science, Peking University, 100871, Beijing, P. R. China}
\email{jnwdyi@163.com}

\author{Zhifei Zhang}
\address{School of Mathematical Science, Peking University, 100871, Beijing, P. R. China}
\email{zfzhang@math.pku.edu.cn}

\date{\today}

\maketitle

\begin{abstract}
In this paper, we solve Gallay's conjecture on the spectral lower bound and pseudospecrtal bound
for the linearized operator of the Navier-Stokes equation in $\R^2$ around rapidly rotating Oseen vortices.
\end{abstract}

\section{introduction}

In this paper, we consider the Navier-Stokes equations in $\R^2$
\ben\label{eq:NS}
\left\{\begin{array}{l}
\pa_t v-\nu\Delta v+v\cdot\na v+\na p=0,\\
\text{div}\,v=0,\\
v(0,x)=v_0(x),
\end{array}\right.
\een
where $v(t,x)$ denotes the velocity, $p(t,x)$ denotes the pressure and $\nu>0$ is the viscosity coefficient.
Let $\om(t,x)=\pa_{2}v^1-\pa_{1}v^2$ be the vorticity. The vorticity formulation of \eqref{eq:NS} takes
\ben\label{eq:NS-w}
\pa_t\om-\nu\Delta \om+v\cdot\na \om=0,\quad \om(0,x)=\om_0(x).
\een
Given the vorticity $\om$, the velocity can be recovered by the Biot-Savart law
\ben
v(t,x)=\f 1 {2\pi}\int_{\R^2}\f {(x-y)^\perp} {|x-y|^2}\om(t,y)dy=K_{BS}\ast \om.
\een

It is well known that the Navier-Stokes equations \eqref{eq:NS-w} has a family of self-similar solutions called Lamb-Oseen vortices of the form
\ben
\om(t,x)=\f \al {\nu t}\cG\Big(\f x {\sqrt{\nu t}}\Big),\quad v(t,x)=\f \al {\sqrt{\nu t}}v^G\Big(\f x {\sqrt{\nu t}}\Big),
\een
where the vorticity profile and the velocity profile are given by
\beno
\cG(\xi)=\f 1 {4\pi}e^{-|\xi|^2/4},\quad v^G(\xi)=\f 1 {2\pi}\f {\xi^\perp} {|\xi|^2}\Big(1-e^{-|\xi|^2/4}\Big).
\eeno
It is easy to see that $\int_{\R^2}\om(t,x)dx=\al$ for any $t>0$. The parameter $\al\in \R$ is called the circulation Reynolds number.

To investigate the long-time behaviour of \eqref{eq:NS-w}, it is convenient to introduce the self-similar variables
\beno
\xi=\f x {\sqrt{\nu t}},\quad \tau=\log t,
\eeno
and the rescaled vorticity $w$ and the rescaled velocity $u$
\beno
\om(t,x)=\f 1 tw\Big(\log t,\f x {\sqrt{\nu t}}\Big),\quad  v(t,x)=\sqrt{\f \nu t}u\Big(\log t, \f x {\sqrt{\nu t}}\Big).
\eeno
Then $(w,u)$ satisfies
\ben\label{eq:NS-Rw}
\pa_\tau w+u\cdot\na w=Lw,
\een
where the linear operator $L$ is given by
\ben
L=\Delta+\f \xi 2\cdot\na+1.
\een
For any $\al\in \R$, the Lamb-Oseen vortex $\al \cG(\xi)$ is a steady solution of \eqref{eq:NS-Rw}.
Gallay and Wayne \cite{GW1, GW2} proved that for the integrable initial vorticity, the long-time behaviour of
the 2-D Navier-Stokes equations can be described by the Lamb-Oseen vortex. More precisely, for any initial data
$w_0\in L^1(\R^2)$, the solution of \eqref{eq:NS-Rw} satisfies
\beno
\lim_{\tau\to+\infty}\big\|w(\tau)-\al \cG\big\|_{L^1(\R^2)}=0,\quad \al=\int_{\R^2}w_0(\xi)d\xi.
\eeno
This result suggests that $\al\cG$ is a stable equilibrium of \eqref{eq:NS-Rw} for any $\al \in \R$.
This situation is very similar to the Couette flow $(y,0)$ in a finite channel, which is stable for any Reynolds number \cite{DR}.
Recently, there are many important works \cite{BM1, BM2, BM3, LZ, WZZ1}
devoted to the study of long-time behaviour of the Navier-Stokes(Euler) equations around the Couette flow.

To study the stability of $\al G$, it is natural to consider the linearized equation around $\al\cG(\xi)$, which takes as follows
\ben
\pa_\tau w=(L-\al\Lambda)w,
\een
where $\Lambda$ is a nonlocal linear operator defined by
\ben\label{def:Lam}
\Lambda w=v^G\cdot\na w+u\cdot\na \cG=\Lambda_1w+\Lambda_2w,\quad u=K_{BS}\ast w.
\een

The operator $L-\al\Lambda$ in the weighted space $Y=L^{2}(\R^2, \cG^{-1}dx)$ defined in section 2 has a compact resolvent. Thus,
the spectrum of $L-\al\Lambda$ in $Y$ is a sequence of eigenvalues $\{\lambda_n(\al)\}_{n\in \N}$ satisfying $\text{Re}\lambda_n(\al)\le 0$ for any $n, \al$.
A very important problem is to study how the spectrum changes as $|\al|\to +\infty$,
which corresponds to the high Reynolds number limit(the most relevant regime for turbulent flows).

The eigenvalues which correspond to the eigenfunctions in the kernel of $\Lambda$ do not change as $\al$ varies. We denote by $L_\perp$ and $\Lambda_\perp$ the restriction of the operators $L$ and $\Lambda$ to the orthogonal complement of $\text{ker}\,\Lambda$ in $Y$.
Then we define the spectral lower bound
\ben
\Sigma(\al)=\inf\Big\{\text{Re}\,z: z\in \sigma\big(-L_\perp+\al\Lambda_\perp\big)\Big\}
\een
and pseudospectral bound
\ben
\Psi(\al)=\Big(\sup_{\lambda\in \R}\|\big(L_\perp-\al\Lambda_\perp-i\lambda\big)^{-1}\|_{Y\to Y}\Big)^{-1}.
\een
For selfadjoint operators, spectral and pesudospectral bounds are the same. Here $L-\al\Lambda$
is a non-selfadjoint operator. It is easy to see that $\Sigma(\al)\ge \Psi(\al)$ for any $\al\in \R$.
In fact, $\Sigma(\al)$ and $\Psi(\al)$ are different. Moreover, the pseudo-spectrum plays
an important role in the hydrodynamic stability \cite{TRD},
and the spectrum theory of non-selfadjoint operator is also a very active topic \cite{Da, DSZ, Tre, TE}.

Maekawa \cite{Mae} proved that $\Sigma(\al)$ and $\Psi(\al)$ tend to infinity as $|\al|\to +\infty$.
However, the proof does not provide explicit bounds on $\Sigma(\al)$ and $\Psi(\al)$.
Numerical calculations performed by Prochazka and Pullin \cite{PP1, PP2} indicate
that $\Sigma(\al)=O(|\al|^\f12)$ as $|\al|\to +\infty$. Based on the analysis for a model problem,
Gallay \cite{Ga} proposed the following conjecture.\smallskip

{\bf Conjecture:}  there exists $C>0$ independent of $\al$ so that as $|\al|\to +\infty$,
\beno
\Sigma(\al)\ge C^{-1}|\al|^\f12,\quad C^{-1}|\al|^\f13\le \Psi(\al)\le C|\al|^\f13.
\eeno

If this conjecture is true, then it shows that the linearized operator $L-\al\Lambda$ becomes highly non-selfadjoint in the fast
rotating limit, and the fast rotation has a strong stabilizing effect on vortices.

To solve this conjecture, Gallagher and Gallay suggested the following model problem (see Villani \cite{Vill} P. 53 and \cite{Vill2}).\smallskip

{\bf Model problem:} identify sufficient condition on $f:\R\to \R$, so that the real parts of the eigenvalues of
\beno
H_\al=-\pa_x^2+x^2+i\al f(x)
\eeno
in $L^2(\R)$ go to infinity as $|\al|\to +\infty$, and estimate this rate.\smallskip

Let $\Sigma(\al)$ be the infimum of the real part of $\sigma(H_\al)$ and $\Psi(\al)^{-1}$ be the supremum of the norm
of the resolvent of $H_\al$ along the imaginary axis. Under the appropriate conditions on $f$, Gallagher, Gallay and Nier \cite{GGN}
proved that $\Sigma(\al)$ and $\Psi(\al)$ go to infinity as $|\al|\to +\infty$,  and presented the precise estimate of the growth rate of $\Psi(\al)$.
Their proof used the hypocoercive method, localization techniques, and semiclassical subelliptic estimates.

For the simplified linearized operator $L-\al \Lambda_1$, Deng \cite{DW1} proved that $\Psi(\al)=O(|\al|^\f13)$.
The same result was proved by Deng \cite{DW2} for the full linearized operator restricted to a smaller subspace than $\text{ker}(\Lambda)^\perp$.
Deng used the multiplier method based on Weyl calculus \cite{Ler}.
\smallskip

The goal of this paper is to give a positive answer on Gallay's conjecture.
The main difficulty comes from the nonlocal operator $\Lambda_2$
so that the hypocoercive method introduced by Villani \cite{Vill} does not work.
In fact, the linearized operator with the nonlocal skew-adjoint operator often appears in the linear stability theory of the incompressible fluids.

In this paper, we develop a method to handle the nonlocal operator. The most key idea is to reduce the nonlocal operator to a model operator by constructing the wave operator. This is motivated by the following simple fact in the scattering theory. Let $A, B$ be two selfadjoint operators in the Hilbert space $H$. Let $U(t)=e^{itA}$ and $V(t)=e^{itB}$ be the strongly continuous   groups of unitary operators. The wave operator is defined by
\beno
W_\pm=\lim_{t\to \pm\infty}W(t),\quad W(t)=U(-t)V(t).
\eeno
Then it holds that
\ben
AW_\pm=W_\pm B.\label{eq:wave}
\een
In fact, we have
\beno
e^{isA}e^{-itA}e^{itB}=e^{-i(t-s)A}e^{i(t-s)B}e^{isB},
\eeno
which gives by taking $t\to \pm \infty$ that
\beno
e^{isA}W_\pm=W_\pm e^{isB}.
\eeno
Then the identity \eqref{eq:wave} follows by taking the derivative in $s$ at $s=0$.

In a joint work of the last two authors and Zhao \cite{WZZ2}, we use similar ideas to prove the optimal enhanced dissipation rate for the linearized Navier-Stokes equations in $\T^2$ around the Kolmogorov flow.

\section{Spectral analysis of the linearized operator}

In this section, we recall some facts about the spectrum of the linearized operator $L-\al \Lambda$ from \cite{GW1, GW2, Ga, GM}. Although these facts will not be used in our proof, they will be helpful to understand this spectral problem.
\smallskip

Let $\rho(\xi)$ be a nonnegative function. We introduce the weighted $L^2$ space
\beno
L^2(\R^2, \rho d\xi)=\Big\{w\in L^2(\R^2): \|w\|_{L^2(\rho)}^2=\int_{\R^2}|w(\xi)|^2\rho(\xi)d\xi<+\infty\Big\},
\eeno
which is a (real) Hilbert space equipped with the scalar product
\beno
\langle w_1, w_2\rangle_{L^2{(\rho)}}=\int_{\R^2}w_1(\xi)w_2(\xi)\rho(\xi)d\xi.
\eeno
We denote $Y=L^2(\R^2, \cG^{-1}d\xi)$.

\begin{lemma}
It holds that
\begin{itemize}
\item[1.] the operator $L$ is selfadjoint in $Y$ with compact resolvent and purely discrete spectrum
\ben\label{eq:L-sp}
\sigma(L)=\big\{-\f n 2: n=0,1,2,\cdots\big\}.
\een

\item[2.] the operator $\Lambda$ is skew-symmetric in Y.
\end{itemize}
\end{lemma}

The first fact follows from the following observation:
\ben\label{def:cL}
\cL=-\cG^{-\f12}L\cG^\f12=-\Delta+\f {|\xi|^2} {16}-\f12
\een
is a two-dimensional harmonic oscillator, which is self-adjoint in $L^2(\R^2)$ with compact resolvent and discrete spectrum
given by $-\sigma(L)$. Furthermore, we know that
\begin{itemize}

\item[1.] $\lambda_0=0$ is a simple eigenvalue of $L$ with the eigenfunction $\cG$;

\item[2.] $\lambda_1=-\f12$ is an eigenvalue of $L$ of multiplicity two with the eigenfunctions $\pa_1 \cG$ and $\pa_2 \cG$;

\item[3.] $\lambda_1=-1$ is an eigenvalue of $L$ of multiplicity three with the eigenfunctions $\Delta \cG, (\pa_1^2-\pa_2^2)\cG$ and $\pa_1\pa_2\cG$.

\end{itemize}

Now we consider the spectrum of $L-\al\Lambda$ in $Y$ for any fixed $\al\in \R$. Since $\Lambda$ is a
relatively compact perturbation of $L$ in $Y$, $L-\al\Lambda$ has a compact resolvent in $Y$ by the classical perturbation theory \cite{Kato}.
So, the spectrum of $L-\al\Lambda$ is a sequence of eigenvalues $\{\lambda_n(\al)\}_{n\in \N}$. Using the fact that
\beno
\Lambda w=0\quad \text{for}\quad w=\cG, \pa_1\cG, \pa_2 \cG, \Delta \cG,
\eeno
we deduce that $0, -\f12, -1$ are also eigenvalues of $L-\al\Lambda$ for any $\al\in \R$. Let us introduce the following subspaces
of $X$:
\beno
&&{Y_0}=\Big\{w\in Y: \int_{\R^2}w(\xi)d\xi=0\Big\}=\big\{\cG\big\}^{\perp},\\
&&{Y_1}=\Big\{w\in {Y_0}: \int_{\R^2}\xi w(\xi)d\xi=0\Big\}=\big\{\cG, \pa_1 \cG, \pa_2 \cG\big\}^{\perp},\\
&&{Y_2}=\Big\{w\in {Y_1}: \int_{\R^2}|\xi|^2w(\xi)d\xi=0\Big\}=\big\{\cG, \pa_1 \cG, \pa_2 \cG, \Delta \cG\big\}^{\perp}.
\eeno
These spaces are invariant under the linear evolution generated by $L-\al \Lambda$.

The following proposition shows that the Oseen vortex $\al\cG$ is spectrally stable in $Y$ for any $\al\in \R$.

\begin{proposition}
For any $\al\in \R$, the spectrum of $L-\al \Lambda$ satisfies
\beno
&&\sigma(L-\al\Lambda)\subset \big\{z\in\C: \text{Re}\, z\le 0\big\}\qquad \text{in}\quad\, Y,\\
&&\sigma(L-\al\Lambda)\subset \big\{z\in\C: \text{Re}\, z\le -\f12\big\}\quad\, \text{in}\quad {Y_0},\\
&&\sigma(L-\al\Lambda)\subset \big\{z\in\C: \text{Re}\, z\le -1\big\}\quad\, \text{in}\quad {Y_1},\\
&&\sigma(L-\al\Lambda)\subset \big\{z\in\C: \text{Re}\, z<-1\big\}\quad\, \text{in}\quad {Y_2}.
\eeno
\end{proposition}

The operator $L-\al\Lambda$ is invariant under rotations with respect to the origin. Thus, it is natural to introduce the polar coordinates $(r,\theta)$
in $\R^2$.  Let us decompose
\ben
Y=\oplus_{n\in\N}X_n,
\een
where $X_n$ denote the subspace of all $w\in Y$
so that
\beno
w(r\cos\theta, r\sin\theta)=a(r)\cos(n\theta)+b(r)\sin(n\theta)
\eeno
for some radial functions $a, b:\R^+\to \R$.

\begin{lemma}
$\text{ker}\,\Lambda=X_0\oplus\big\{\al\pa_1\cG+\beta\pa_2\cG\big\}$. In particular, $\text{ker}\,\Lambda^\perp\subset{\oplus_{n>0}X_n}$.
\end{lemma}

\section{Reduction to one-dimensional operators}

Following Deng's work \cite{DW2}, we reduce the linearized operator to a family of one-dimensional operators.

We conjugate the linearized operator $L-\al \Lambda$ with $\cG^\f12$, and then obtain a linear operator $\cH_\al$ in $L^2(\R^2, d\xi)$:
\ben\label{def:H}
\cH_\al=-\cG^{-\f12}L\cG^\f12+\al \cG^{-\f12}\Lambda \cG^{\f12}=\cL+\al\cM,
\een
where $\cL$ is defined by \eqref{def:cL} and $\cM$ is defined by
\beno
\cM w=v^G\cdot \na w-\f12 \cG^\f12\xi\cdot\big(K_{BS}\ast (\cG^\f12w)\big).
\eeno

Let us introduce some notations:
\begin{align}
&\mathcal{K}_k[h]=\f{1}{2|k|}\int_{0}^{+\infty}\min(\f rs,\f sr)^{|k|}sh(s)ds,\\
&\sigma(r)=\f{1-e^{-r^2/4}}{r^2/4},\quad g(r)=e^{-r^2/8}.
\end{align}
Then for $w=\sum_{k\in \mathbb{Z}^{*}}w_k(r)e^{ik\theta}$, we have
\begin{align*}
\big((\mathcal{H}_{\alpha}-i\lambda)w\big)(r\cos\theta,r\sin\theta)=\sum_{k\in\mathbb{Z}^{*}}(\mathcal{H}_{\alpha,k,\lambda}w_{k})(r)e^{ik\theta},
\end{align*}
where the operator $\mathcal{H}_{\alpha,k,\lambda}$ acts on $L^2(\mathbb{R}_+, rdr)$ and is given by
\begin{align}
&\mathcal{H}_{\alpha,k,\lambda}=-\partial_r^2-\f{1}{r}\partial_r+\f{k^2}{r^2}+\f{r^2}{16}-\f12+i\beta_k(\sigma(r)-\nu_k)-i\beta_kg\mathcal{K}_k[g\cdot],
\end{align}
where
\ben
\beta_k=\f{\alpha k}{8\pi},\quad \lambda=\beta_k\nu_k\in \R.
\een
Without loss of generality, we assume $|\beta_k|\ge 1$ for any $|k|\le 1$.

We introduce the operator
\ben
\widetilde{\mathcal{H}}_{\alpha,k,\lambda}=r^{\f12}\mathcal{H}_{\alpha,k,\lambda} r^{-\f12}:=\widetilde{\mathcal{H}}_{k}.
\een
Then $\widetilde{\mathcal{H}}_k$ acts on $L^2(\mathbb{R}_+, dr)$ and is given by
\begin{align}
&\widetilde{\mathcal{H}}_k=-\partial^2_r+\f{k^2-\f14}{r^2}+\f{r^2}{16}-\f12+i\beta_k(\sigma(r)-\nu_k)-i\beta_k g\widetilde{\mathcal{K}}_k[g\cdot],\\
&\widetilde{\mathcal{K}}_k[h]=\f{1}{2|k|}\int_{0}^{+\infty}\min(\f rs,\f sr)^{|k|}(rs)^{\f12}h(s)ds,
\end{align}
and $C_0^{\infty}(\R_+)$ is a core of the operator $\widetilde{\mathcal{H}}_k$ with domain
\begin{align}
D(\widetilde{\mathcal{H}}_k)=\big\{\omega\in H^2_{loc}(\mathbb{R}_+,dr)\cap L^2(\mathbb{R}_+,dr):\widetilde{\mathcal{H}}_k\omega\in L^2(\mathbb{R}_+,dr)\big\}.
\end{align}
That is,
\begin{align*}
&D=D(\widetilde{\mathcal{H}}_k)=\Big\{w\in L^2(\mathbb{R}_+,dr):\partial_r^2w,\f{w}{r^2},r^2w\in L^2(\mathbb{R}_+,dr) \Big\}\quad |k|\geq2,\\
&D_1=D(\widetilde{\mathcal{H}}_{{k}})=\Big\{w\in L^2(\mathbb{R}_+,dr): r^{\f12}\partial_r^2 ({w}/{r^{\f12}}),r^{\f12}\partial_r({w}/{r^{\f32}}),r^2w\in L^2(\mathbb{R}_+,dr)\Big\}\quad |k|=1.
\end{align*}
Then the resolvent estimate is reduced to the following estimate
\begin{align}
\big\|\widetilde{\mathcal{H}}_ku \big\|_{L^2(\mathbb{R}_+,dr)}\gtrsim |\beta_k|^{\frac{1}{3}}\big\|u \big\|_{L^2(\mathbb{R}_+,dr)}.\nonumber
\end{align}

We also write
\ben
\widetilde{\mathcal{H}}_k=\widetilde{A}_k+i\beta_k\widetilde{B}_k-i\lambda,
\een
where
\begin{align*}
&\widetilde{A}_k=-\partial_r^2+\f{k^2-\f14}{r^2}+\f{r^2}{16}-\f12,\\
&\widetilde{B}_k=\sigma(r)-g\widetilde{\mathcal{K}}_k[g\cdot].
\end{align*}
It is easy to see that
\begin{align}\label{Ker-B}
\text{Ker}(\widetilde{B}_1)={\rm span}\big\{r^{\f32}g(r)\big\},\quad \text{Ker}(\widetilde{B}_k)=\big\{0\big\} \quad \text{for}\,\, |k|\geq2.
\end{align}
Thus, $L-\al\Lambda|_{(\ker \Lambda)^{\bot}}$ is unitary equivalent to $\bigoplus\limits_{|k|=1}\widetilde{\mathcal{H}}_k|_{(\ker \widetilde{B}_1)^{\bot}}\oplus\bigoplus\limits_{|k|\geq 2}\widetilde{\mathcal{H}}_k.$

In the sequel, we denote by $\langle \cdot,\cdot\rangle$ the $L^2(\R_+, dr)$ inner product, and by $\|\cdot\|$ the norm of $L^2(\R_+,dr)$, {$\|\cdot\|_{L^p}$ the norm of $L^p(\R_+,dr)$}.
The notation $a\gtrsim b$ or $a\lesssim b$ means that there exists a constant $C>0$ independent of $\al, k,\lambda$ so that
\beno
a\ge C^{-1}b \quad \text{or}\quad a\le Cb.
\eeno

\section{Resolvent estimate of $\widetilde{\mathcal{H}}_1$}

As $\overline{\widetilde{\mathcal{H}}_{-1}w}=\widetilde{\mathcal{H}}_{1}\overline{w}$, it is enough to prove the following resolvent estimate for $\widetilde{\mathcal{H}}_1$.

\begin{theorem}\label{thm:H1}
For any $\lambda\in\mathbb{R}$ and $w \in\{r^{\f32}g(r)\}^{\perp}\cap D_1$, we have
\begin{align*}
\|\widetilde{\mathcal{H}}_1w\|\gtrsim |\beta_1|^{\f13}\|w\|.
\end{align*}
Moveover, there exist $\lambda\in \R$ and $v \in\{r^{\f32}g(r)\}^{\perp}\cap D_1$ so that
\begin{align*}
\|\widetilde{\mathcal{H}}_1v\|\lesssim |\beta_1|^{\f13}\|v\|.
\end{align*}
\end{theorem}

\subsection{Reduction to the model operator $\cL_1$}

Let us introduce the operator $T$ defined by
\begin{align}
Tw(r)=w(r)+\f{I_1[w](r)g(r)}{\sigma'(r)r^{\f32}},
\end{align}
where
\begin{align}
I_1[w](r)=\int_{0}^{r}s^{\f32}g(s)w(s)ds.
\end{align}
It is easy to check that $T$ is a bounded linear operator in $L^2(\mathbb{R}_+, dr)$. The adjoint operator $T^*$ is also a bounded linear operator in $L^2(\mathbb{R}_+, dr)$  given by
\begin{align}
T^*\omega(r)=\omega(r)+r^{\f32}g(r)\int_{r}^{+\infty}\f{\omega(s)g(s)}{s^{\f32}\sigma'(s)}ds.
\end{align}

\begin{lemma}\label{lem:T-p1}
It holds that
\begin{itemize}

\item[1.]  $\|T w\|^2=\|w\|^2-\f{\langle w,r^\f32g\rangle^2}{\|r^\f32g\|^2}$;

\item[2.] $T^*T=P$, where $P$ is the projection to $\{r^{\f32}g(r)\}^{\perp}\cap L^2(\mathbb{R}_+,dr)$;

\item[3.] $TT^*=I_{L^2(\mathbb{R}_+,dr)}$.

\end{itemize}
\end{lemma}

\begin{proof}
The first one is equivalent to the second one. Thanks to
\begin{align*}
-(r^3\sigma'(r))'=r^3g(r)^2\quad\text{and}\quad I_1[s^{\f32}g(s)](r)=\int_{0}^{r}s^3g(s)^2ds=-r^3\sigma'(r),
\end{align*}
we find that
\beno
T(r^{\f32}g)(r)=r^{\f32}g(r)+\f{I_1[s^{\f32}g(s)](r)g(r)}{\sigma'(r)r^{\f32}}=0.
\eeno
Thus, it suffices to check that for any $w \in\{r^{\f32}g(r)\}^{\perp}\cap L^2(\mathbb{R}_+, dr),\, u\in {C_0^{\infty}(\mathbb{R}_+)}$,
\begin{align*}
\langle T^*Tw,u \rangle=\langle w,u\rangle,
\end{align*}
which is equivalent to {verifying} that
\begin{align}
\Big\langle \f{I_1[w](r)g(r)}{\sigma'(r)r^{\f32}},u\Big\rangle+\Big\langle w, \f{I_1[u](r)g(r)}{\sigma'(r)r^{\f32}}\Big\rangle
+\Big\langle \f{I_1[w](r)g(r)}{\sigma'(r)r^{\f32}},\f{I_1[u](r)g(r)}{\sigma'(r)r^{\f32}}\Big\rangle=0.\nonumber
\end{align}
Using the facts that for $w\in \big\{r^{\f32}g(r)\big\}^{\perp}$,
\begin{align*}
I_1[\omega](0)=\lim\limits_{r\to+\infty}I_1[\omega](r)=0,
\end{align*}
and $(r^3\sigma'(r))'=-r^3g(r)^2$, we get by integration by parts that
\begin{align*}
\Big\langle \f{I_1[w](r)g(r)}{\sigma'(r)r^{\f32}},{u} \Big\rangle=&\int_{0}^{+\infty}\f{I_1[w]}{\sigma'(r)r^3}dI_1[\overline{u}]\\
=&-\int_{0}^{+\infty}I_1[\overline{u}](\f{I_1[w]}{\sigma'(r)r^3})'dr\\
=&-\int_{0}^{+\infty}I_1[\overline{u}]\f{r^{\f32}g(r)w(r)\sigma'(r)r^3-I_1[\omega](\sigma'(r)r^3)'}{(\sigma'(r)r^3)^2}{dr}\\
=&-\int_{0}^{+\infty}w(r)\f{I_1[\overline{u}]g(r)}{\sigma'(r)r^{\f32}}dr-\int_{0}^{+\infty}I_1[w]I_1[\overline{u}]\f{g^2(r)}{(\sigma'(r)r^{\f32})^2}dr.
\end{align*}
This shows that $T^*T=P$.

On the other hand, we have
\begin{align*}
|T^*w|^2-|w|^2=-\partial_r(r^3\sigma'(r)|f_1|^2),\quad f_1(r)=\int_r^{+\infty}\f{w(s)g(s)}{s^{\f32}\sigma'(s)}ds,
\end{align*}
which gives $\|T^*w\|^2=\|w\|^2$, thus $TT^*=I$.
\end{proof}

We have the following  important relationship between $T$ and $\widetilde{B}_1$.
\begin{lemma}\label{lem:TB-relation}
It holds that
\begin{align}
T\widetilde{B}_1=\sigma(r)T.\nonumber
\end{align}
\end{lemma}

\begin{proof}
Direct calculation gives
\begin{align*}
T\widetilde{B}_1w&=\widetilde{B}_1w+\f{I_1[\widetilde{B}_1w]g(r)}{\sigma'(r)r^{\f32}}\\
&=\sigma(r)w-g\widetilde{\mathcal{K}}_1[gw]+\f{I_1[\sigma w]g(r)}{\sigma'(r)r^{\f32}}-\f{I_1[g\widetilde{\mathcal{K}}_1[gw]]g(r)}{\sigma'(r)r^{\f32}}.
\end{align*}
Thus, it suffices to show that
\begin{align}
I_1[\sigma w]=\sigma'(r)r^{\f32}\widetilde{\mathcal{K}}_1[g w]+I_1[g\widetilde{\mathcal{K}}_1[gw]]+\sigma(r)I_1[w].\label{eq:I1-1}
\end{align}
Direct calculation shows that
\begin{align*}
\sigma'(r)r^{\f32}\widetilde{\mathcal{K}}_1[gw]&=r^{\f12}(e^{-\f{r^2}{4}}-\sigma(r))\int_{0}^{+\infty}\min(\f{r}{s},\f{s}{r})(rs)^{\f12}g(s)w(s)ds\\
&=r^{\f12}(e^{-\f{r^2}{4}}-\sigma(r))[\int_{0}^{r}r^{-\f12}s^{\f32}g(s)w(s)ds+\int_{r}^{+\infty}r^{\f32}s^{-\f12}g(s)w(s)ds]\\
&=(e^{-\f{r^2}{4}}-\sigma(r))I_1[w]+r^2(e^{-\f{r^2}{4}}-\sigma(r))\int_{r}^{+\infty}s^{-\f12}g(s)w(s)ds,
\end{align*}
and
\begin{align*}
&I_1[g\widetilde{\mathcal{K}}_1[gw]](r)=\int_{0}^{r}s^{\f32}g^2(s)\widetilde{\mathcal{K}}_1[gw](s)ds\\
&=\int_{0}^{r}s^{\f32}g^2(s)ds\f{1}{2}\int_{0}^{+\infty}\min(\f{t}{s},\f{s}{t})(ts)^{\f12}g(t)w(t)dt\\
&=\f12\int_{0}^{r}se^{-\f{s^2}{4}}ds\int_{0}^{s}t^{\f32}g(t)w(t)dt+\f{1}{2}\int_{0}^{r}s^3e^{-\f{s^2}{4}}ds\int_{s}^{+\infty}t^{-\f12}g(t)w(t)dt\\
&=\f12\int_{0}^{r}t^{\f32}g(t)w(t)dt\int_{t}^{r}se^{-\f{s^2}{4}}ds+\f12\int_{0}^{r}s^3e^{-\f{s^2}{4}}ds\int_{s}^{r}t^{-\f12}g(t)w(t)dt\\
&\quad+\f12\int_{0}^{r}s^3e^{-\f{s^2}{4}}ds\int_{r}^{+\infty}t^{-\f12}g(t)w(t)dt\\
&=\int_{0}^{r}t^{\f32}g(t)w(t)(e^{-\f{t^2}{4}}-e^{-\f{r^2}{4}})dt+\f12\int_{0}^{r}t^{-\f12}g(t)w(t)dt\int_{0}^{t}s^3e^{-\f{s^2}{4}}ds\\
&\quad+[4(1-e^{-\f{r^2}{4}})-r^2e^{-\f{r^2}{4}}]\int_{r}^{+\infty}t^{-\f12}g(t)w(t)dt\\
&=-e^{-\f{r^2}{4}}I_1[w]+I_1[\sigma w]+r^2(\sigma(r)-e^{-\f{r^2}{4}})\int_{r}^{+\infty}t^{-\f12}g(t)w(t)dt,
\end{align*}
which give \eqref{eq:I1-1}.
\end{proof}

\begin{lemma}\label{lem:com-TA1}
It holds that
\begin{align}
[T,\widetilde{A}_{1}]w=T\widetilde{A}_1w-\widetilde{A}_1Tw=f(r)Tw,\nonumber
\end{align}
where
\beno
f(r)=2\f{g(r)^4}{(\sigma'(r)^2)}+\f{g(r)^2}{\sigma'(r)}\big(\f{6}{r}-r\big)\geq 0.
\eeno
\end{lemma}

\begin{proof}
First of all, we have
\begin{align*}
[T,\widetilde{A}_{1}]w=\f{I_1[(-\partial_r^2+\f34\f{1}{r^2}+\f{r^2}{16})w]g(r)}{\sigma'(r)r^\f{3}{2}}
-\Big(-\partial_r^2+\f34\f{1}{r^2}+\f{r^2}{16}\Big)\Big(\f{I_1[w]g(r)}{\sigma'(r)r^\f32}\Big).
\end{align*}
Using the facts that
\begin{align*}
&I_1[-\partial^2_rw]=-r^{\f32}g(r)w'(r)+(r^\f32g(r))'w(r)-\int_{0}^{r}(s^\f32g(s))''w(s)ds,\\
&(-\partial_r^2+\f34\f{1}{r^2}+\f{r^2}{16})r^\f32g(r)=r^\f32g(r),
\end{align*}
we deduce that
\begin{align*}
I_1\big[(-\partial_r^2+\f34\f{1}{r^2}+\f{r^2}{16})w\big]=-r^\f32g(r)\omega'(r)+(r^\f32g(r))'\omega(r)+I_1[\omega],
\end{align*}
Direct calculation gives
\begin{align*} \partial^2_r\Big(\f{I_1[\omega]g(r)}{\sigma'(r)r^\f32}\Big)=\omega'(r)\f{g^2(r)}{\sigma'(r)}+\omega(r)\big(\f{g^2(r)}{\sigma'(r)}\big)'
+r^\f32g(r)\omega(r)\big(\f{g(r)r^\f32}{\sigma'(r)r^3}\big)'+I_1[\omega]\big(\f{g(r)r^\f32}{\sigma'(r)r^3}\big)''.
\end{align*}
Let $F=r^\f32g(r)$ and ${G(r)}=\sigma'(r)r^3$. We have
\begin{align*}
F'=\big(\f32\cdot\f{1}{r}-\f{r}{4}\big)F,\quad F''=\big(\f34\cdot\f{1}{r^2}-1+\f{r^2}{16}\big)F,\quad G'=-F^2.
\end{align*}

Summing up, we obtain
\begin{align*}
[T,\widetilde{A}_{1}]w=&\big[\f{F'F}{G}+(\f{F^2}{G})'+F(\f{F}{G})'\big]w\\
&+\big(1-\f34\cdot\f{1}{r^2}-\f{r^2}{16}\big)\f{I_1[w]F}{G}+I_1[w](\f{F}{G})''.
\end{align*}
On the other hand, we have
\begin{align*}
&\big(\f{F}{G}\big)''=\big(\f{3}{4}\cdot\f{1}{r^2}-1+\f{r^2}{16}\big)\f{F}{G}+\f{F}{G}\big(4\f{F'F}{G}+2\f{(G')^2}{G^2}\big),\\
&\f{F'F}{G}+(\f{F^2}{G})'+F\big(\f{F}{G}\big)'=4\f{F'F}{G}+2\f{(G')^2}{G^2},
\end{align*}
Then we infer that
\begin{align*}
[T,\widetilde{A}_{1}]w&=\big(4\f{F'F}{G}+2\f{(G')^2}{G^2}\big)(w+\f{I_1[w]F}{G})\\
&=\big(2\f{g^4}{(\sigma')^2}+\f{g^2}{\sigma'(r)}(\f{6}{r}-r)\big)Tw=f(r)Tw.
\end{align*}

It remains to prove that $f(r)\geq0$. We have
\begin{align*}
2\f{g^4}{(\sigma')^2}+\f{g^2}{\sigma'(r)}(\f{6}{r}-r)=\f{r^6+r^2(6-r^2)(r^2+4-4e^{\f{r^2}{4}})}{32(\f{r^2}{4}+1-e^{\f{r^2}{4}})^2},
\end{align*}
while by Taylor expansion, we have
\begin{align*}
r^6+r^2(6-r^2)(r^2+4-4e^{\f{r^2}{4}})=&2r^4+24r^2-24r^2e^{\f{r^2}{4}}+4r^4e^{\f{r^2}{4}}\\
=&2r^2\Big[r^2+2r^2\sum_{n=0}^{+\infty}\f{1}{n!}(\f{r^2}{4})^n-12\sum_{n=1}^{+\infty}\f{1}{n!}(\f{r^2}{4})^n\Big]\\
=&2r^2\Big[8\sum_{n=2}^{+\infty}\f{1}{(n-1)!}(\f{r^2}{4})^n-12\sum_{n=2}^{+\infty}\f{1}{n!}(\f{r^2}{4})^n\Big]\geq0.
\end{align*}
This completes the proof.
\end{proof}

It follows from Lemma \ref{lem:TB-relation} and Lemma \ref{lem:T-p1} that for $w\in\{r^{\f32}g(r)\}^{\perp}\cap D_1$,
\begin{align*}
T\widetilde{\mathcal{H}}_1w&=T\widetilde{A}_1w+i\beta_1T\widetilde{B}_1w-i\lambda Tw\\
&=T\widetilde{A}_1T^*Tw+i\beta_1\sigma(r)Tw-i\lambda Tw.
\end{align*}
Lemma \ref{lem:T-p1} ensures that $T:\big\{r^{\f32}g(r)\big\}^{\perp}\to L^2(\R_+,dr)$ is invertible and $T^{-1}=T^*$.
Let $w=T^{-1}u$. We infer from Lemma \ref{lem:com-TA1} that
\begin{align}
T\widetilde{\mathcal{H}}_1T^{-1}u=&T\widetilde{A}_1T^{-1}u+i\beta_1\sigma(r)u-i\lambda u\nonumber\\
=&\widetilde{A}_1u+f(r)u+i\beta_1\sigma(r)u-i\lambda u={\mathcal{L}}_1u,\label{eq:H1toL1}
\end{align}
where
\ben\label{def:f}
f(r)=2\f{g(r)^4}{(\sigma(r)')^2}+\f{g(r)^2}{\sigma(r)'}\big(\f{6}{r}-r\big).
\een
So, the operator $T$ plays a role of wave operator. Let \begin{align*}
D({\mathcal{L}}_1)=\big\{\omega\in H^2_{loc}(\mathbb{R}_+,dr)\cap L^2(\mathbb{R}_+,dr):{\mathcal{L}}_1\omega\in L^2(\mathbb{R}_+,dr)\big\}.
\end{align*}
Then $ u\in D({\mathcal{L}}_1)\Leftrightarrow T^{*}u\in D({\mathcal{H}}_1)\cap\big\{r^{\f32}g(r)\big\}^{\perp},$ and $D({\mathcal{L}}_1)=D(\widetilde{\mathcal{H}}_3)=D. $

Moreover, we have
\beno
\langle \widetilde{\cH}_1w, w\rangle=\langle \widetilde{\cH}_1T^{-1}u, T^*u\rangle=\langle \cL_1u, u\rangle.
\eeno
On the other hand, $\|w\|=\|Tw\|=\|u\|$ for any  $w\in\{r^{\f32}g(r)\}^{\perp}\cap D_1$.
Thus, we reduce the resolvent estimate of $\widetilde{\cH}_1$ to one of the model operator $\cL_1$.

\subsection{Coercive estimates}

\begin{lemma}\label{lem:coer-A1}
The operator $\widetilde{A}_1$ can be represented as
\begin{align}
\big(\widetilde{A}_1-\f12\big)w=-r^{-\f32}g^{-1}\partial_r\big[r^3g^2\partial_r(r^{-\f32}g^{-1}w)\big].
\end{align}
In particular, we have
\begin{align}
\widetilde{A}_k\geq\f12\quad\text{for}\quad k\ge 1.
\end{align}
\end{lemma}

\begin{proof} Let $F(r)=r^{\f32}g(r)$. Then we have
\begin{align*}
-r^{-\f32}g^{-1}\partial_r\big[r^3g^2\partial_r(r^{-\f32}g^{-1}w)\big]=&-F^{-1}\partial_r\big[F^2\partial_r(F^{-1}w)\big]\\
=&\big(-\partial_r^2+\f{F''}{F}\big)w=(-\partial_r^2+\f34\cdot\f{1}{r^2}-1+\f{r^2}{16})w=\big(\widetilde{A}_1-\f12\big)w,
\end{align*}
here we used $F''=(\f34\cdot\f{1}{r^2}-1+\f{r^2}{16})F$.

Then for any $w\in D$, we have
\begin{align*}
\big\langle(\widetilde{A}_1-\f12)w,w\big\rangle=&-\big\langle F^{-1}\partial_r\big[F^2\partial_r(F^{-1}w)\big],w\big\rangle
=\big\|F\partial_r(F^{-1}\omega) \big\|^2\geq 0.
\end{align*}
This shows that $\widetilde{A}_k\geq\widetilde{A}_1\geq \f12$.
\end{proof}

\begin{lemma}\label{lem:coer-A1+f}
It holds that
\begin{align}
\widetilde{A}_1+f(r)\gtrsim \f{1}{r^2}+r^2.
\end{align}
\end{lemma}

\begin{proof}
By the proof of Lemma \ref{lem:TB-relation}, we know that
\begin{align*}
f(r)=&r^2\f{\Big\{\sum\limits_{n=2}^{+\infty}(\f{2}{(n-1)!}-\f{3}{n!})(\f{r^2}{4})^n\Big\}}{4(\f{r^2}{4}+1-e^{r^2/4})^2}\geq r^2\f{\sum\limits_{n=2}^{+\infty}\f{1}{n!}(\f{r^2}{4})^n}{4(\f{r^2}{4}+1-e^{r^2/4})^2}\\
\geq&\f{r^2}{4(e^{\f{r^2}{4}}-1-\f{r^2}{4})}.
\end{align*}

Let $h(r)=\f{3}{4}\cdot\f{1}{r^2}+\f{r^2}{16}-\f12+\f{r^2}{4(e^{\f{r^2}{4}}-1-\f{r^2}{4})}$. Then there exists $\e_0\in(0,1)$ so that $h(r)\gtrsim \f{1}{r^2}$ for $r<\e_0$ and $h(r)\gtrsim {r^2}$ for $r>\f{1}{\e_0},$ and $h(r)$ can attain its minimum.
Thus, if $h(r)>0$, $h(r)$ has a positive lower bound. For this, let $u=\f{r^2}{4}$. Then by Taylor's expansion, we get
\begin{align*}
h(r)=&\f{3}{16}\cdot\f{1}{u}+\f{1}{4}u-\f12+\f{u}{e^u-1-u}\\
=&\f{\f{3}{16}\sum\limits_{n=2}^{+\infty}\f{1}{n!}u^n+\f14u^2\sum\limits_{n=2}^{+\infty}\f{1}{n!}u^n-\f12u\sum\limits_{n=2}^{+\infty}\f{1}{n!}u^n+u^2}{u(e^u-1-u)}\\
=&\f {\sum\limits_{n=2}^{+\infty} a_nu^n}{u(e^u-1-u)},
\end{align*}
where
\begin{align*}
&a_2=\f{35}{32},\quad a_3=-\f{7}{32},\quad a_4=\f{19}{384},\quad 2\sqrt{a_2a_4}>|a_3|,\\
&a_n=\f{1}{n!}\big(\f{3}{16}+\f{n(n-1)}{4}-\f{n}{2}\big)>0(n\geq5).
\end{align*}
Hence, there exists $c_0>0$ such that $h(r)\geq c_0$. So, there exists $C>0$ such that for any $r\in[\e_0,\f{1}{\e_0}]$, we have
$h(r)\geq C(\f{1}{r^2}+r^2).$

Summing up, we conclude that
\begin{align*}
\widetilde{A}_1+f(r)\ge \f{3}{4r^2}+\f{r^2}{16}-\f12+f(r)\ge h(r)\gtrsim \f{1}{r^2}+r^2.
\end{align*}
The proof is completed.
\end{proof}

\subsection{Resolvent estimate of $\cL_1$}

In this subsection, we prove Theorem \ref{thm:H1}. It suffices to show that for any $u=Tw, w\in\{r^{\f32}g(r)\}^{\perp}\cap D_1$,
\ben
\|\cL_1u\|\gtrsim |\beta_1|^\f13\|u\|.
\een
The proof is split into three cases.
\smallskip

{\bf Case 1.} $\nu_1\geq1$

Be Lemma \ref{lem:coer-A1+f}, we get
\begin{align*}
|\langle \cL_1u,u \rangle|&\sim \langle (\tilde{A}_1+f)u,u \rangle+|\beta_1|\langle(\nu_1-\sigma(r))u,u \rangle\\
&\gtrsim \langle (\f{1}{r^2}+r^2)u,u \rangle+|\beta_1|\langle(1-\sigma(r))u,u \rangle.
\end{align*}
Using the fact that
\beno
1-\sigma(r)=1-\f{1-e^{-\f{r^2}{4}}}{r^2/4}\sim r^2(r\to 0),\quad \lim\limits_{r\to\infty}1-\sigma(r)=1,
\eeno
we deduce that
\begin{align*}
&\int_0^1\big[\f{1}{r^2}+|\beta_1|(1-\sigma(r))\big]|u|^2dr\gtrsim \int_0^1 (\f{1}{r^2}+|\beta_1|r^2)|u|^2dr\gtrsim \int_0^1|\beta_1|^{\f12}|u|^2dr,\\
&\int_1^{+\infty}\big[\f{1}{r^2}+r^2+|\beta_1|(1-\sigma(r))\big]|u|^2dr\gtrsim\int_{1}^{+\infty}(1+|\beta_1|)|u|^2dr\gtrsim \int_1^{+\infty}|\beta_1|^{\f12}|u|^2dr,
\end{align*}
which show that for $\nu_1\ge 1$,
\begin{align}
|\langle \cL_1u,u \rangle|\gtrsim |\beta_1|^{\f12}\|u\|^2.
\end{align}

{\bf Case 2}. $\nu_1\leq0$

In this case, we have by Lemma \ref{lem:coer-A1+f} that
\begin{align*}
|{\langle \cL_1}u,u \rangle|&\sim \langle (\tilde{A}_1+f)u,u \rangle
+|\beta_1|\langle(\sigma(r)-\nu_1)u,u \rangle\\
&\gtrsim \langle (\f{1}{r^2}+r^2)u,u \rangle+|\beta_1|\langle \sigma(r)u,u \rangle.
\end{align*}
Thanks to $\lim\limits_{r\to0}\sigma(r)=1$ and $\sigma(r)\sim\f{1}{r^2}(r\to\infty)$, we infer that
\begin{align*}
&\int_0^1\big[\f{1}{r^2}+r^2+|\beta_1|\sigma(r)\big]|u|^2dr\gtrsim\int_0^1(1+|\beta_1|)|u|^2dr\gtrsim \int_0^1|\beta_1|^{\f12}|u|^2dr,\\
&\int_1^{+\infty}(r^2+|\beta_1|\sigma(r))|u|^2dr\gtrsim\int_{1}^{+\infty}(r^2+\f{1}{r^2}|\beta_1|)|u|^2dr\gtrsim \int_1^{+\infty}|\beta_1|^{\f12}|u|^2dr,
\end{align*}
which shows that for $\nu_1\le 0$,
\begin{align}
|\langle \cL_1u,u\rangle|\gtrsim|\beta_1|^{\f12}\|u\|^2.
\end{align}

{\bf Case 3.} $0<\nu_1<1$

Let $\nu_1=\sigma(r_1)$ for some $r_1>0$. We split this case into two subcases:
\begin{align*}
 |\beta_1|\leq\max\big(\f{1}{r_1^4},r^6_1\big)\quad \text{and}\quad|\beta_1|\geq\max\big(\f{1}{r_1^4},r^6_1\big).
\end{align*}

\begin{lemma}
If $|\beta_1|\leq\max\big(\f{1}{r_1^4},r^6_1\big)$, then we have
\begin{align}
\|{\mathcal{L}}_1u\|\gtrsim|\beta_1|^{\f13}\|u\|.\nonumber
\end{align}
\end{lemma}

\begin{proof}
If $|\beta_1|\leq1$, then $\f{1}{r^2}+r^2\geq1\geq |\beta_1|^{\f13}$. Lemma \ref{lem:coer-A1+f} gives
\begin{align*}
{|\langle{\mathcal{L}}_1u,u\rangle|\gtrsim \big\langle(r^2+\f{1}{r^2})u,u\big\rangle\geq|\beta_1|^{\f13}\|u\|^2}.
\end{align*}
If $1\leq|\beta_1|\leq\max(\f{1}{r_1^4},r^6_1)$, we only need to check the following cases
\begin{align*}
&r_1\leq1,\quad 1\leq|\beta_1|\leq\f{1}{r_1^4}\Longrightarrow \|{\mathcal{L}}_1u\|\gtrsim|\beta_1|^{\f12}\|u\|,\\
&r_1\geq1,\quad 1\leq|\beta_1|\leq r_1^4\Longrightarrow \|{\mathcal{L}}_1u\|\gtrsim|\beta_1|^{\f12}\|u\|,\\
&r_1\geq1,\quad r_1^4\leq|\beta_1|\leq r_1^6\Longrightarrow\|{\mathcal{L}}_1u\|\gtrsim|\beta_1|^{\f13}\|u\|.
\end{align*}
By Lemma \ref{lem:coer-A1+f} again, we have
\begin{align*}
 |\langle{\mathcal{L}}_1u,u\rangle|\gtrsim \big\langle(r^2+\f{1}{r^2})u,u\big\rangle +|\beta_1| |\langle(\nu_1-\sigma(r))u,u\rangle|,
\end{align*}
which along with Lemma \ref{lem:beta-med} gives our results.
\end{proof}

\begin{lemma}
If $|\beta_1|\geq\max(\f{1}{r_1^4},r^6_1)$, then we have
\begin{align}
\|{\mathcal{L}}_1u\|\gtrsim|\beta_1|^{\f13}\|u\|.\nonumber
\end{align}
\end{lemma}

\begin{proof}Let $\delta>0$ be so that $\delta^3|\beta_1|\min(r_1,r^{-3}_1)=1$. Thanks to $|\beta_1|\geq\max(\f{1}{r_1^4},r^6_1)$, we get
\begin{align*}
|\beta_1|^{-\f12}\leq \min\big(r_1^2,\f 1 {r_1^3}\big).
\end{align*}
Thus, we have
\begin{align*}
\delta^3|\beta_1|^{\f12}\leq r_1\quad \text{for}\,\,r_1\leq1,\quad \delta^3|\beta_1|^{\f12}\leq 1\quad\text{for}\,\,r_1\geq 1,
\end{align*}
which in particular give $\delta^2|\beta_1|^{\f13}\leq1$. Also we have $0<\delta\leq\min(r_1,\f{1}{r_1})$. Hence, it suffices to show that
\begin{align}
\|u\|\lesssim\delta^2\|{\mathcal{L}}_1u\|.\label{eq:L1-coer1}
\end{align}

Let us choose $r_{-}\in(r_1-\delta,r_1)$ and $r_+\in(r_1,r_1+\delta)$ so that
\begin{align}
|u'(r_{-})|^2+|u'(r_+)|^2\leq\f{\big\|u' \big\|^2}{\delta}.\label{eq:u-cut}
\end{align}
We get by integration by parts that
\begin{align*}
&\rm{Re}\big\langle{\mathcal{L}}_1u,i \text{sgn}(\beta_1)(\chi_{(0,r_-)}-\chi_{(r_+,+\infty)})u\big\rangle\\
&=\rm{Re}\langle-\partial_r^2u+i\beta_1(\sigma-\nu_1)u, i \text{sgn}(\beta_1)(\chi_{(0,r_-)}-\chi_{(r_+,+\infty)})u\rangle\\
&=\rm{Re}\left(\int_0^{r_-}(-i\text{sgn}(\beta_1)|\partial_ru|^2
+|\beta_1|(\sigma-\nu_1)|u|^2)dr+i\text{sgn}(\beta_1)(u'\overline{u})(r_-)\right)\\
&\quad+\rm{Re}\left(\int_{r_+}^{+\infty}(i\text{sgn}(\beta_1)|\partial_ru|^2
+|\beta_1|(\nu_1-\sigma)u|^2)dr+i\text{sgn}(\beta_1)(u'\overline{u})(r_+)\right)\\
&\geq\int_0^{r_-}|\beta_1|(\sigma-\nu_1)|u|^2dr+\int_{r_+}^{+\infty}|\beta_1|(\nu_1-\sigma)|u|^2dr
-|(u'\overline{u})(r_-)|-|(u'\overline{u})(r_+)|.
\end{align*}
Due to $0<\delta\leq\min(r_1,\f{1}{r_1})$, $0<r_1-\delta<r_1+\delta\leq2r_1$. Then we get by Lemma \ref{lem:sigma} that
\begin{align*}
&\sigma(r)-\nu_1\geq\sigma(r_1-\delta)-\sigma(r_1)\gtrsim\delta|\sigma'(r_1)|\quad 0<r<r_1-\delta,\\
&\nu_1-\sigma(r)\geq\sigma(r_1)-\sigma(r_1+\delta)\gtrsim\delta|\sigma'(r_1)|\quad r>r_1+\delta,
\end{align*}
from which and \eqref{eq:u-cut}, we infer that
\begin{align*}
&\rm{Re}\langle{\mathcal{L}}_1u,i \text{sgn}(\beta_1)(\chi_{(0,r_-)}-\chi_{(r_+,+\infty)})u\rangle\\
&\geq \int_0^{r_-}|\beta_1|(\sigma-\nu_1)|u|^2dr+\int_{r_+}^{+\infty}|\beta_1|(\nu_1-\sigma)|u|^2dr
-|(u'\overline{u})(r_-)|-|(u'\overline{u})(r_+)|\\
&{\geq C^{-1}} |\beta_1\delta\sigma'(r_1)|\big\|u \big\|^2_{L^2(\mathbb{R}_+\setminus(r_1-\delta,r_1+\delta))}-\f{{2}}{\delta^{\f12}}\big\|u' \big\|_{L^2}\big\| u\big\|_{L^{\infty}}.
\end{align*}
Thanks to $\sigma'(r)=\f{2}{r}(e^{-\f{r^2}{4}}-\f{1-e^{-\f{r^2}{4}}}{r^2/4})$, we have $|\sigma'(r)|\sim \f{1}{r^3}(r\to\infty)$ and $|\sigma'(r)|\sim r(r\to0)$. Thus, $|\sigma'(r)|\sim\min(r,\f{1}{r^3})$. Recall that $\delta^3|\beta_1|\min(r_1,r_1^{-3})=1$. Then $|\beta_1\delta^3\sigma'(r_1)|\sim1$. Thus, we obtain
\begin{align*}
\|u\|^2_{L^2(\mathbb{R}_+\setminus(r_1-\delta,r_1+\delta))}\lesssim \delta^2\|u\|\|{\mathcal{L}}_1u\|_{L^2}+\delta^{\f32}\|u'\|\|u\|_{L^{\infty}}.
\end{align*}
On the other hand, it is obvious that
\beno
\|u'\|^2\leq \|u\|\|{\mathcal{L}}_1u\|_{L^2},\quad \|u\|_{L^\infty}\le \|u\|^\f12\|u'\|^\f12.
\eeno
Consequently, we deduce that
\begin{align*}
\|u\|^2=&\|u\|^2_{L^2(\mathbb{R}_+\setminus(r_1-\delta,r_1+\delta))}+\|u\|^2_{L^2(r_1-\delta,r_1+\delta)}\\
\lesssim& \delta^2\|u\|\|{\mathcal{L}}_1u\|+\delta^{\f32}\|u'\|\|u\|_{L^{\infty}}+\delta\|u\|_{L^{\infty}}^2\\
\lesssim&\delta^2\|u\|\|{\mathcal{L}}_1u\|+\delta^2\|u'\|^2+\delta\|u\|_{L^{\infty}}^2\\
\leq &\delta^2\|u\|\big\|{\mathcal{L}}_1u\|+\delta^2\|u\|\|{\mathcal{L}}_1u\|+\delta\|u'\|\|u\|\\
\lesssim &\|u\|(\delta^2\|{\mathcal{L}}_1u\|)+\|u\|^{\f32}(\delta^2\|{\mathcal{L}}_1u\|)^{\f12},
\end{align*}
which implies \eqref{eq:L1-coer1}.
\end{proof}

\subsection{Sharpness of pseudospectral bound}

Finally, let us prove the sharpness of the pseudospectral bound of $\widetilde{\cH}_1$. That is, there exist
$\lambda\in\mathbb{R}$ and $ v\in\{r^{\f32}g(r)\}^{\perp}\cap D_1$, such that
\begin{align}\label{eq:psudo}
\|\widetilde{\mathcal{H}}_1v\|\leq C |\beta_1|^{\f13}\|v\|.
\end{align}
Take $\lambda\in \R$ so that $|\beta_1|= r_1^6\geq1$. We take $u(r)=\eta(r_1(r-r_1))$, where
 $\eta(r)=r^2(r-1)^2$ for ${0<r<1}$, $\eta(r)=0$ for ${r(r-1)\geq 0}$ .
Then we have
\beno
\|u\|=r_1^{-\f12}\|\eta\|,\quad  \|\partial_r^2u\|=r_1^{\f32}\|\partial_r^2\eta\|\leq Cr_1^2\|u \|.
\eeno
By Lemma \ref{lem:sigma}, we have
\begin{align}
|\beta_1(\sigma(r)-\sigma(r_1))|\sim |\beta_1\sigma'(r_1)||r-r_1|\leq C\f{|\beta_1|}{r_1^4}\leq C r_1^2,\quad \text{for}\,\,|r-r_1|\leq\f{1}{r_1},\nonumber
\end{align}
and we also have that for ${0<r-r_1<\f{1}{r_1}}$,
\begin{align}
\big|(\f{3}{4r^2}+\f{r^2}{16}-\f12+f)u\big|\leq Cr_1^2 |u|.\nonumber
\end{align}
Thus, we can conclude that
\begin{align*}
\|{\mathcal{L}}_1u \|\leq Cr_1^2\|u \|,
\end{align*}
which along with Lemma \ref{lem:T-p1} gives
\begin{align*}
\|\widetilde{\mathcal{H}}_1T^*u \|=\|T\widetilde{\mathcal{H}}_1T^*u \|=&\|{\mathcal{L}}_1u\|\leq C|\beta_1|^{\f13}\|u\|=C|\beta_1|^{\f13}\|T^*u \|.
\end{align*}
This gives \eqref{eq:psudo} by taking $v=T^*u$.

\section{Resolvent estimate of $\widetilde{\mathcal{H}}_k, k\ge 2$}

In this section, we will prove the following resolvent estimate for $\widetilde{\mathcal{H}}_k, k\ge 2$.

\begin{theorem}\label{thm:Hk}
Let $k\ge 2$. For any $\lambda\in\mathbb{R}$ and $w\in D$, we have
\begin{align}
\|\widetilde{\mathcal{H}}_kw\big\|\gtrsim |\beta_k|^{\f13}\|w\|.
\end{align}
\end{theorem}

\subsection{Coercive estimates of $\widetilde{A}_k$ and $\widetilde{B}_k$}

\begin{lemma}\label{lem:HktoLk}
For any $|k|\geq1$ and $w \in L^2(\mathbb{R}_+;dr)$, we have
\begin{align*}
&\langle (I-\widetilde{B}_k)w,w\rangle\geq\int_{0}^{+\infty}(1-\sigma(r))|w|^2dr,\\
&\langle \widetilde{B}_kw,w\rangle\geq(1-\f{1}{|k|})\int_{0}^{+\infty}\sigma(r)|w|^2dr.
\end{align*}
\end{lemma}

\begin{proof}
Let us first prove that the operator $g\widetilde{\mathcal{K}}_{k}[g\cdot]$ is nonnegative. For this, we write
\begin{align*}
\widetilde{\mathcal{K}}_{k}[w](r)&=\f{1}{2|k|}\int_{0}^{+\infty}\min(\f{r}{s},\f{s}{r})^{|k|}(rs)^{\f12}w(s)ds\\
&=\f{1}{2|k|}\int_{0}^{r}r^{\f12-|k|}s^{\f12+|k|}w(s)ds+\f{1}{2|k|}\int_{r}^{+\infty}r^{\f12+|k|}s^{\f12-|k|}w(s)ds.
\end{align*}
Then we find that
\begin{align*} &(\widetilde{\mathcal{K}}_{k}[w](r))'=\f{\f12-|k|}{2|k|}\int_{0}^{r}(\f{s}{r})^{|k|+\f12}w(s)ds+\f{\f12+|k|}{2|k|}\int_{r}^{+\infty}(\f{r}{s})^{|k|-\f12}w(s)ds,\\
&(\widetilde{\mathcal{K}}_{k}[w](r))''=\f{k^2-\f14}{r^2}\tilde{\mathcal{K}}_k[w](r)-w(r).
\end{align*}
In particular, we find that
\ben
\big(-\partial_r^2+\f{k^2-\f14}{r^2}\big)\widetilde{\mathcal{K}}_k[w](r)=w(r).\label{eq:K-eigen}
\een
Using the following pointwise estimates of $\widetilde{\mathcal{K}}_{k}[w](r)$
\begin{align*}
&|\widetilde{\mathcal{K}}_{k}[w](r)|\leq\f{1}{2|k|}\int_{0}^{+\infty}\min(r,s)|w(s)|ds\leq\f{1}{2|k|}\min\big(r\|w\|_{L^1},\|rw \|_{L^1}\big),\\
&|\partial_r(\widetilde{\mathcal{K}}_{k}[w](r))|\leq\int_{0}^{+\infty}\f{\min(r,s)}{r}|w(s)|ds\leq\min\big(\|w\|_{L^1},\f{1}{r}\|rw\|_{L^1}),
\end{align*}
we infer that
\begin{align*}
\widetilde{\mathcal{K}}_k[w](\widetilde{\mathcal{K}}_k[w])'\Big|_{r=0,+\infty}=0.
\end{align*}
Then we get by using \eqref{eq:K-eigen} and integration by parts that
\begin{align*}
\langle g\widetilde{\mathcal{K}}_{k}[gw],w\rangle&=\langle\widetilde{\mathcal{K}}_{k}[gw],gw\rangle=\langle \widetilde{\mathcal{K}}_{k}[gw],(-\partial_r^2+\f{k^2-\f14}{r^2})\widetilde{\mathcal{K}}_k[gw]\rangle\\
&=\|\partial_r(\widetilde{\mathcal{K}}_k[gw])\|^2+(k^2-\f14)\|\f{\widetilde{\mathcal{K}}_k[gw]}{r}\|^2\geq0.
\end{align*}

Next we give a upper bound for $g\widetilde{\mathcal{K}}_{k}[g\cdot]$.
\begin{align*}
\Big|\int_{0}^{+\infty}g\widetilde{\mathcal{K}}_k[gw]\overline{w(r)} dr\Big|&\leq\f{1}{2|k|}\int_{0}^{+\infty}\int_{0}^{+\infty}\min(\f{r}{s},\f{s}{r})^{|k|}(rs)^{\f12}|g(r)w(s)||g(s)w(r)|dsdr\\
&\leq\f{1}{4|k|}\int_0^{+\infty}\int_{0}^{+\infty}\min(\f{r}{s},\f{s}{r})(rs)^{\f12}[(\f{r}{s})^{\f32}g^2(r)|w(s)|^2+(\f{s}{r})^{\f32}g^2(s)|w(r)|^2]dsdr\\
&=\f{1}{|k|}\int_0^{+\infty}\widetilde{\mathcal{K}}_1[r^{\f32}g^2(r)](s)\f{|w(s)|^2}{s^{\f32}}ds=\f{1}{|k|}\int_{0}^{+\infty}\sigma(s)|w(s)|^2ds,
\end{align*}
which gives
\begin{align*}
0\leq g\widetilde{\mathcal{K}}_{k}[g\cdot]\leq\f{1}{|k|}\sigma(r).
\end{align*}
As a consequence, we deduce that
\begin{align*}
&\langle(1-\widetilde{B}_k)w,w\rangle=\langle(1-\sigma)w+g\widetilde{\mathcal{K}}_{k}[gw],w\rangle\geq\langle(1-\sigma)w,w\rangle,\\
&\langle\widetilde{B}_kw,w\rangle=\langle\sigma w,w\rangle-\langle g\tilde{\mathcal{K}}_{k}[gw],w\rangle\geq(1-\f{1}{|k|})\langle\sigma w,w\rangle.
\end{align*}
The proof is completed.
\end{proof}

The following lemma gives a sharper lower bound of $\widetilde{A}_k$ than Lemma \ref{lem:coer-A1}.

\begin{lemma}\label{lem:coer-Ak}
Let $k\ge 2$. Then for any $w\in D$, we have
\begin{align*}
\langle\widetilde{A}_kw,w\rangle\gtrsim \big\langle(\f{k^2}{r^2}+r^2)w,w\big\rangle.
\end{align*}
\end{lemma}

\begin{proof}
For $|k|\geq2$, we have
\begin{align*}
\widetilde{A}_k&\geq\f{k^2-\f14}{r^2}+\f{r^2}{16}-\f12\\
&=\Big(\f23\cdot\f{k^2-\f14}{r^2}+{\f{3}{32}}\cdot\f{r^2}{k^2-\f14}-\f12\Big)+\f13\cdot \f{k^2-\f14}{r^2}+\Big(\f{1}{16}-{\f{3}{32}}\cdot\f{1}{k^2-\f14}\Big)r^2\\
&\geq \f13\cdot \f{k^2-\f14}{r^2}+(\f{1}{16}-{\f{3}{32}}\cdot\f{1}{k^2-\f14})r^2\sim \f{k^2}{r^2}+r^2,
\end{align*}
which gives our result.
\end{proof}

\subsection{Resolvent estimate for $\nu_k\ge 1$ or $\nu_k\le 0$}

In this subsection, we prove Theorem \ref{thm:Hk} for the case of $\nu_k\ge 1$ or $\nu_k\le 0$.\smallskip

First of all, for $\nu_k\ge 1$, we infer from Lemma \ref{lem:HktoLk} that
\begin{align*}
|\langle\widetilde{\mathcal{H}}_kw,w\rangle|&\sim\langle\widetilde{A}_kw,w\rangle+|\beta_k|\langle(\nu_k-\widetilde{B}_k)w,w\rangle\\
&\geq \langle\widetilde{A}_kw,w\rangle+|\beta_k|\langle(1-\widetilde{B}_k)w,w\rangle\\
&\gtrsim \big\langle(\f{1}{r^2}+r^2)w,w\big\rangle+|\beta_k|\langle(1-\sigma(r))w,w\rangle.
\end{align*}
Thanks to $1-\sigma(r)=1-\f{1-e^{-\f{r^2}{4}}}{r^2/4}\sim r^2(r\to0)$ and $\lim\limits_{r\to\infty}1-\sigma(r)=1$, we get
\begin{align*}
&\int_0^1\big[\f{1}{r^2}+|\beta_k|(1-\sigma(r))\big]|w(r)|^2dr\gtrsim\int_0^1 (\f{1}{r^2}+|\beta_k|r^2)|w(r)|^2dr\gtrsim \int_0^1|\beta_k|^{\f12}|w(r)|^2dr,\\
&\int_1^{+\infty}\big[\f{1}{r^2}+r^2+|\beta_k|(1-\sigma(r))\big]|w(r)|^2dr\gtrsim\int_{1}^{+\infty}(1+|\beta_k|)|w(r)|^2dr\gtrsim \int_1^{+\infty}|\beta_k|^{\f12}|w(r)|^2dr,
\end{align*}
which yield that for $\nu_k\geq1$,
\begin{align}
|\langle\widetilde{\mathcal{H}}_kw,w\rangle|\gtrsim \big\langle(\f{1}{r^2}+r^2)w,w\big\rangle+|\beta_k|\langle(1-\sigma(r))w,w\rangle\gtrsim|\beta_k|^{\f12}\|w\|^2
\end{align}

For $\nu_k\leq0$, we infer from Lemma \ref{lem:HktoLk} that
\begin{align*}
|\langle\widetilde{\mathcal{H}}_kw,w\rangle|&\sim\langle\widetilde{A}_kw,w\rangle+|\beta_k|\langle(\widetilde{B}_k-\nu_k)w,w\rangle\\
&\gtrsim \langle\widetilde{A}_kw,w\rangle+|\beta_k|\langle\widetilde{B}_kw,w\rangle\\
&\gtrsim \big\langle(\f{1}{r^2}+r^2)w,w\big\rangle+|\beta_k|\langle\sigma(r)w,w\rangle.
\end{align*}
Thanks to $\lim\limits_{r\to0}\sigma(r)=1$ and $\sigma(r)\sim\f{1}{r^2}(r\to\infty)$, we get
\begin{align*}
&\int_0^1 \big[\f{1}{r^2}+r^2+|\beta_k|\sigma(r)\big]|w(r)|^2dr\gtrsim\int_0^1(1+|\beta_k|)|w(r)|^2dr\gtrsim \int_0^1|\beta_k|^{\f12}|w(r)|^2dr,\\
&\int_1^{+\infty}(r^2+|\beta_k|\sigma(r))|w(r)|^2dr\gtrsim\int_{1}^{+\infty}(r^2+\f{1}{r^2}|\beta_k|)|w(r)|^2dr\gtrsim \int_1^{+\infty}|\beta_k|^{\f12}|w(r)|^2dr,
\end{align*}
which show that for $\nu_k\leq0$
\begin{align}
|\langle\widetilde{\mathcal{H}}_kw,w\rangle|\gtrsim \big\langle(\f{1}{r^2}+r^2)w,w\big\rangle+|\beta_k|\langle\sigma(r)w,w\rangle\gtrsim|\beta_k|^{\f12}\|w\|^2.
\end{align}

\subsection{Resolvent estimate for $0<\nu_k<1$}

In this subsection, we prove Theorem \ref{thm:Hk} for the case of $0<\nu_k<1$.\smallskip

Let $\nu_k=\sigma(r_k)$ for some $r_k>0$. We again divide the proof into two cases:
\begin{align*}
|\beta_k|\leq\max\big(\f{|k|^3}{r_k^4},|k|^3,r_k^6\big)\quad\text{and}\quad |\beta_k|\geq\max\big(\f{|k|^3}{r_k^4},|k|^3,r_k^6\big).
\end{align*}

\subsubsection{Case 1. $|\beta_k|\leq\max\big(\f{|k|^3}{r_k^4},|k|^3,r_k^6\big)$.}

By Lemma \ref{lem:HktoLk}, it suffices to prove the following lemma.

\begin{lemma}
If $|\beta_k|\leq\max\big(\f{|k|^3}{r_k^4},|k|^3,r_k^6\big)$, then we have
\begin{align*}
\|{\widetilde{\mathcal{H}}_k}w\|\gtrsim |\beta_k|^{\f13}\|w\|.
\end{align*}
\end{lemma}

\begin{proof}
If $|\beta_k|\leq |k|^3$, then $\f{k^2}{r^2}+r^2\geq |k|\geq |\beta_k|^{\f13}$. Thus,
\begin{align*}
|\langle\widetilde{\mathcal{H}}_kw,w\rangle|\gtrsim \big\langle(\f{k^2}{r^2}+r^2)w,w\big\rangle\geq |\beta_k|^{\f13}\|w\|^2.
\end{align*}
If $k^3\leq|\beta_k|\leq\max(\f{|k|^3}{r_k^4},r^6_k)$, we only need to check the following cases
\begin{align*}
&r_k\leq 1,\quad |k|^3\leq|\beta_k|\leq\f{|k|^3}{r_k^4}\Longrightarrow\|\widetilde{\mathcal{H}}_kw\|\gtrsim|\beta_k|^{\f12}\|w\|,\\
&r_k\geq\sqrt{k},\quad|k|^3\leq|\beta_k|\leq r_k^4\Longrightarrow\|\widetilde{\mathcal{H}}_kw\|\gtrsim|\beta_k|^{\f12}\|w\|,\\
&r_k\geq\sqrt{k},\quad r_k^4\leq|\beta_k|\leq r_k^6\Longrightarrow \|\widetilde{\mathcal{H}}_kw\|\gtrsim|\beta_k|^{\f13}\|w\|,
\end{align*}
which can be deduced from the following fact
\begin{align*}
 |\langle\widetilde{\mathcal{H}}_kw,w\rangle|\gtrsim \big\langle(r^2+\f{k^2}{r^2})w,w\big\rangle+|\beta_k| {\max(\langle(\nu_k-\sigma(r))w,w\rangle,\langle(\sigma(r)/2-\nu_k)w,w\rangle,0)}
\end{align*}
and Lemma \ref{lem:beta-high}.
\end{proof}

\subsubsection{Case 2. $|\beta_k|\geq\max(\f{|k|^3}{r_k^4},|k|^3,r_k^6)$.}

Let us introduce the operator
\begin{align}
\widetilde{\mathcal{K}}_k^{(r_k)}[f](r)=\int_{0}^{r_k}\widetilde{K}_k^{(r_k)}(r,s)f(s)ds,
\end{align}
where for $0\le r,s\le r_k$,
\beno
\widetilde{K}_k^{(r_k)}(r,s)=\f{1}{2|k|}\min(\f{r}{s},\f{s}{r})^{|k|}(rs)^{\f12}-\f{1}{2|k|}(\f{rs}{r^2_k})^{|k|}(rs)^{\f12}\ge 0.
\eeno
Let $u(r)=\widetilde{\mathcal{K}}_k^{(r_k)}[w](r)$. Then $u\in H^1_0(0,r_k)$ is the unique solution to
\begin{align*}
\Big(-\partial_r^2+\f{k^2-\f14}{r^2}\Big)u=w\quad \text{in}\quad (0,r_k).
\end{align*}

\begin{lemma}\label{lem:K}
It holds that
\begin{align*}
{\rm{Re}}\int_{0}^{r_k}g(r)\widetilde{\mathcal{K}}_k^{(r_k)}[gw](r)\bar{f}(r)dr\leq\f{2}{|k|+1}\int_{0}^{+\infty}(\sigma(s)-\nu_k)|w(s)|^2ds.
\end{align*}
\end{lemma}

\begin{proof}
{As $(r^3\sigma'(r))'=-r^3g(r)^2$  and $g^2$ is decreasing, we have $-r^3\sigma'(r)\geq r^4g^2(r)/4, $ and}
\begin{align*}
\Big(-\partial_r^2+\f{k^2-\f14}{r^2}\Big)\big(r^{|k|+\f12}(\sigma(r)-\nu_k)\big)
&=r^{|k|+\f12}(-\partial_r^2\sigma-(2|k|+1)r^{-1}\partial_r\sigma)\\
&=r^{|k|+\frac{1}{2}}(g^2-(2|k|-2)r^{-1}\partial_r\sigma)\\
&\geq r^{|k|+\frac{1}{2}}(g^2+\f{|k|-1}{2}g^2)\\
&=r^{|k|+\frac{1}{2}} \f{|k|+1}{2}g^2,
\end{align*}
which implies that
\begin{align*}
r^{|k|+\frac{1}{2}}(\sigma-\nu_k)&=\widetilde{\mathcal{K}}_k^{(r_k)}\big[r^{|k|+\frac{1}{2}}(g^2-(2|k|-2)\f{\partial_r\sigma}{r})\big]\\
&\geq \f{|k|+1}{2}\widetilde{\mathcal{K}}_k^{(r_k)}\big[r^{|k|+\frac{1}{2}}g^2\big].
\end{align*}
Therefore, we obtain
\begin{align*}
&{\rm Re}\int_0^{r_k}g\widetilde{\mathcal{K}}_k^{(r_k)}[gw](r)\overline{w}(r)dr\\
&\leq\int_0^{r_k}\int_0^{r_k}
\widetilde{K}_k^{(r_k)}(r,s)|g(r)w(s)||g(s)w(r)|dsdr\\
&\leq\frac{1}{2}\int_0^{r_k}\int_0^{r_k}\widetilde{K}_k^{(r_k)}(r,s)\left(\left(\frac{r}{s}\right)^{|k|+\frac{1}{2}}g^2(r)|w(s)|^2+\left(\frac{s}{r}\right)^{|k|+\frac{1}{2}}g^2(s)|f(r)|^2\right)dsdr\\
&=\int_0^{r_k}\widetilde{\mathcal{K}}_k^{(r_k)}[r^{|k|+\frac{1}{2}}g^2](s)\frac{|w(s)|^2}{s^{|k|+\frac{1}{2}}}ds\leq\frac{2}{|k|+1}\int_0^{r_k}(\sigma(s)-\nu_k){|w(s)|^2}ds.
\end{align*}
This completes the proof.
\end{proof}

To proceed, we introduce the following decomposition: let $\psi=\widetilde{\mathcal{K}}_k[gw]$ and decompose
\ben\label{eq:psi-dec}
\psi(r)=\psi_1(r)+\psi_2(r),
\een
where $\psi_2(r)$ is given by
\begin{align*}
\psi_2(r)=\left\{
\begin{aligned}
&(\f{r}{r_k})^{|k|+\f12}\psi(r_k),\quad 0< r< r_k,\\
&(\f{r}{r_k})^{-|k|+\f12}\psi(r_k),\quad r>r_k.
\end{aligned}
\right.
\end{align*}
Then we find that $\psi_1(r)\in H^1_0(0,r_k)$ and solves
\begin{align}
\Big(-\partial_r^2+\f{k^2-\f14}{r^2}\Big)\psi_1=gw,\quad r\in\mathbb{R}_+\backslash \{r_k\}.\label{eq:psi-ODE}
\end{align}
Thus, $\psi_1(r)=\widetilde{\mathcal{K}}_k^{(r_k)}[gw](r)$ in $(0,r_k)$.

Let $\delta>0$ be such that
\begin{align}
\delta^3|\beta_k|\min(r_k, r_k^{-3})=1.\nonumber
\end{align}
Due to $|\beta_k|\geq\max(\f{k^3}{r_k^4}, |k|^3, r_k^6)$, we have
\begin{align}
0<\delta\leq\min(\f{r_k}{|k|}, \f{1}{r_k}),\label{p.2}
\end{align}
Due to $|\sigma'(r)|\sim\min(r,\f{1}{r^3})$, we also have
\begin{align}
|\beta_k\delta^3\sigma'(r_k)|\sim1.\label{eq:delta-sigma}
\end{align}
We denote
\begin{align*}
\cE(w)=&\frac{\|w'\|\|w\|_{L^{\infty}}}{|\b_k\d^{\frac{3}{2}}\sigma'(r_k)|}+\frac{\|\widetilde{\mathcal{H}}_kw\|\|w\|}{|\b_k\d\sigma'(r_k)|}+
\frac{\|w'\overline{w}\|_{L^{\infty}(r_k-\d,r_k+\d)}}{|\b_k\d\sigma'(r_k)|}+\d\|w\|_{L^{\infty}}^2\\ &+\frac{|\psi(r_k){J}(r_k)|}{\min(1,r_k^2)}+\frac{g(r_k)^2|\psi(r_k)|^2}{\d|\sigma'(r_k)|^2}+\frac{|\psi(r_k)|^2}{r_k^5+1}=\cE_1(w)+\cdots+\cE_7(w),
\end{align*}
where
\begin{align*}
J(r)=\int_{0}^{r}(\f{s}{r})^{|k|+\f12}g(s)w(s)ds-\int_{r}^{+\infty}(\f{r}{s})^{|k|-\f12}g(s)w(s)ds.
\end{align*}
It is easy to see that
\begin{align}
\partial_r\psi=-\f{J}{2}+\f{\psi}{4|k|r}.\label{eq:psi-J}
\end{align}

\begin{lemma}\label{lem:w-L2}
It holds that for any $w\in D$,
\begin{align*}
&\|w\|^2\leq C\cE(w).
\end{align*}
\end{lemma}

\begin{proof}Recall that
\begin{align*}
 \widetilde{\mathcal{H}}_kw=-\partial_r^2w+\Big(\dfrac{k^2-1/4}{r^2}+\dfrac{r^2}{16}-\dfrac{1}{2}\Big)w
+i\b_k((\sigma-v_k)w-g\psi).
\end{align*}
Then we get by integration by parts that
\begin{align*}
&{\rm{Re}}\langle \widetilde{\mathcal{H}}_kw,i\text{sgn}(\b_k)(\chi_{(0,r_k)}-\chi_{(r_k,+\infty)})w\rangle\\
&={\rm{Re}}\langle -\partial_r^2\omega+i\b_k((\sigma-\nu_k)w-g\psi),i\text{sgn}(\b_k)(\chi_{(0,r_k)}-\chi_{(r_k,+\infty)})w\rangle\\
&={\rm{Re}}\left(\int_0^{r_k}(-i\text{sgn}(\b_k)|\partial_rw|^2+|\b_k|(\sigma-v_k)|w|^2-|\b_k|g\psi\overline{w})dr+i\text{sgn}(\b_k)\omega'\overline{w}(r_k)\right)\\
&\quad+{\rm{Re}}\left(\int_{r_k}^{+\infty}(i\text{sgn}(\b_k)|\partial_rw|^2+|\b_k|(\nu_k-\sigma)|w|^2+|\b_k|g\psi\overline{w})dr+i\text{sgn}(\b_k)w'\overline{w}(r_k)\right)\\
&\geq|\beta_k|\int_0^{+\infty}|\sigma-\nu_k||\omega|^2dr-|\beta_k|{\rm Re}\left(\int_0^{r_k}g\psi\overline{\omega}dr-\int_{r_k}^{+\infty}g\psi\overline{w}dr\right)
-2|w'\overline{w}(r_k)|.
\end{align*}
Using \eqref{eq:psi-dec}, we write
\begin{align*}
&\int_0^{r_k}g\psi\overline{w}dr-\int_{r_k}^{+\infty}g\psi\overline{w}dr\\
&=\int_0^{r_k}g\psi_1\overline{w}dr-\int_{r_k}^{+\infty}g\psi_1\overline{w}dr +\int_0^{r_k}g\psi_2\overline{w}dr-\int_{r_k}^{+\infty}g\psi_2\overline{w}dr\\
&=\int_0^{r_k}g\psi_1\overline{w}dr-\int_{r_k}^{+\infty}\psi_1(-\partial_r^2+(k^2-1/4)r^{-2})\overline{\psi}_1dr\\
&\quad+\psi(r_k)\int_0^{r_k}\left(\dfrac{r}{r_k}\right)^{|k|+\frac{1}{2}}g\overline{w}dr
-\psi(r_k)\int_{r_k}^{+\infty}\left(\dfrac{r}{r_k}\right)^{-|k|+\frac{1}{2}}g\overline{w}dr\\
=&\int_0^{r_k}g\psi_1\overline{w}dr-\int_{r_k}^{+\infty}\big(|\partial_r\psi_1|^2+(k^2-1/4)r^{-2}|\psi_1|^2\big)dr+\psi(r_k)\overline{J}(r_k).
\end{align*}
By \eqref{eq:psi-ODE}, we have
\begin{align*}
\int_0^{r_k}g\psi_1\overline{w}dr=\int_0^{r_k}\psi_1(-\partial_r^2+\f{k^2-1/4}{r^2})\overline{\psi}_1dr=\int_0^{r_k}\big(|\partial_r\psi_1|^2+\f{k^2-1/4}{r^2}|\psi_1|^2\big)dr.
\end{align*}
We get by Lemma \ref{lem:K} that
\begin{align*}
&\int_0^{r_k}g\psi_1\overline{w}dr=\int_0^{r_k}g\widetilde{\mathcal{K}}_k^{(r_k)}[gw](r)\overline{w}dr\leq\frac{2}{|k|+1}\int_0^{r_k}(\sigma(s)-\nu_k){|w(s)|^2}
ds.
\end{align*}
Then we conclude that
\begin{align*}
&{\rm Re}\langle\widetilde{\mathcal{H}}_kw,i\text{sgn}(\b_k)(\chi_{(0,r_k)}-\chi_{(r_k,+\infty)})w\rangle\\ &\geq|\b_k|\left(\frac{|k|+1}{2}\int_0^{r_k}g\psi_1\overline{w}dr-\int_0^{r_k}g\psi_1\overline{w}dr
+\int_{r_k}^{+\infty}(|\partial_r\psi_1|^2+\f{k^2-1/4}{r^2}|\psi_1|^2)dr\right)\\
&\quad-|\b_k\psi(r_k)\overline{J(r_k)}|-2|w'\overline{w}(r_k)|\\
&=|\b_k|\left(\frac{|k|-1}{2}\int_0^{r_k}g\psi_1\overline{s}dr+\int_{r_k}^{+\infty}(|\partial_r\psi_1|^2+\f{k^2-1/4}{r^2}|\psi_1|^2)dr\right)-|\b_k\psi(r_k)\overline{J(r_k)}|-2|\omega'\overline{\omega}(r_k)|\\
&=|\b_k|\left(\frac{|k|-1}{2}\int_0^{r_k} (|\partial_r\psi_1|^2+\f{k^2-1/4}{r^2}|\psi_1|^2)dr+\int_{r_k}^{+\infty}(|\partial_r\psi_1|^2+\f{k^2-1/4}{r^2}|\psi_1|^2)dr\right)\\
&\quad-|\b_k\psi(r_k)\overline{J(r_k)}|-2|w'\overline{\omega}(r_k)|\\
&\geq |\b_k|\f{A_1}{2}-|\b_k\psi(r_k)\overline{J(r_k)}|-2|w'\overline{w}(r_k)|,
\end{align*}
where
\begin{align*}
A_1=\int_{0}^{+\infty}(|\partial_r\psi_1|^2+(k^2-1/4)r^{-2}|\psi_1|^2)dr.
\end{align*}
This shows that
\begin{align}\label{eq:A1-est}
A_1\leq 2|\psi(r_k){J(r_k)}|+\f{1}{|\beta_k|}\big(4|w'\overline{w}(r_k)|+2\|w\|\|\widetilde{\mathcal{H}}_k\omega\|\big).
\end{align}

Similarly, we have
\begin{align*}
&{\rm Re}\langle\widetilde{\mathcal{H}}_kw,i\text{sgn}(\b_k)\frac{\chi_{\R_+\setminus(r_k-\d,r_k+\d)}}{\sigma-\nu_k}w\rangle\\
&={\rm Re}\big\langle -\partial_r^2w
+i\b_k((\sigma-\nu_k)w-g\psi),i\text{sgn}(\b_k)\frac{\chi_{\R_+\setminus(r_k-\d,r_k+\d)}}{\sigma-\nu_k}w\big\rangle\\
&={\rm Re}\int_{\R_+\setminus(r_k-\d,r_k+\d)}\left(-i\frac{\text{sgn}(\b_k)|w'|^2}{\sigma-\nu_k}+i\frac{\text{sgn}(\b_k)w'\overline{w}\sigma'}{(\sigma-\nu_k)^2}
+|\b_k||w|^2-\frac{|\b_k|g\psi\overline{w}}{\sigma-\nu_k}\right)dr\\
&\quad+{\rm Re}\Big(i\text{sgn}(\b_k)\frac{w'\overline{w}}{\sigma-\nu_k}(r_k-\d)
-i\text{sgn}(\b_k)\frac{w'\overline{w}}{\sigma-\nu_k}(r_k+\d)\Big)\\
&\geq-\|w'\|\|w\|_{L^{\infty}}\left\|\frac{\sigma'}{(\sigma-\nu_k)^2}\right\|_{L^2(\R_+\setminus(r_k-\d,r_k+\d))}+|\b_k|\int_{\R_+\setminus(r_k-\d,r_k+\d)}
|w|^2dr\\
&\quad-\f{|\beta_k|}{2}\int_{\R_+\setminus(r_k-\d,r_k+\d)}\Big(
|w|^2+\frac{g^2|\psi|^2}{(\sigma-\nu_k)^2}\Big)dr-\|w'\overline{w}\|_{L^{\infty}(r_k-\d,r_k+\d)}\Big(\frac{1}{\sigma(r_k-\d)-\nu_k}+\frac{1}{\nu_k-\sigma(r_k+\d)}\Big)\\
&=\f{|\beta_k|}{2}\int_{\R_+\setminus(r_k-\d,r_k+\d)}\Big(
|w|^2-\frac{g^2|\psi|^2}{(\sigma-v_k)^2}\Big)dr -\|w'\|_{L^2}\|w\|_{L^{\infty}}\left\|\frac{\sigma'}{(\sigma-\nu_k)^2}\right\|_{L^2(\R_+\setminus(r_k-\d,r_k+\d))}\\
&\quad-\|w'\overline{w}\|_{L^{\infty}(r_k-\d,r_k+\d)}\Big(\frac{1}{\sigma(r_k-\d)-\nu_k}+\frac{1}{\nu_k-\sigma(r_k+\d)}\Big),
\end{align*}
which gives
\begin{align*}
\|w\|_{L^2(\R_+\setminus(r_k-\d,r_k+\d))}^2\leq&\f{2}{|\beta_k|}\|\widetilde{\mathcal{H}}_kw\|\|w\|\left\|\frac{1}{\sigma-\nu_k}\right\|_{L^{\infty}(\R_+\setminus(r_k-\d,r_k+\d))}\\
&+\int_{\R_+\setminus(r_k-\d,r_k+\d)}\frac{g^2|\psi|^2}{(\sigma-\nu_k)^2}dr+\f{2}{|\beta_k|}\|w'\|\|w\|_{L^{\infty}}\left\|\frac{\sigma'}{(\sigma-\nu_k)^2}\right\|_{L^2(\R_+\setminus(r_k-\d,r_k+\d))}\\
&+\f{2}{|\beta_k|}\|w'\overline{w}\|_{L^{\infty}(r_k-\d,r_k+\d)}\Big(\frac{1}{\sigma(r_k-\d)-\nu_k}+\frac{1}{\nu_k-\sigma(r_k+\d)}\Big).
\end{align*}

By Lemma \ref{lem:sigma}, we have
\begin{align*}
\left\|\frac{1}{\sigma-\nu_k}\right\|_{L^{\infty}(\R_+\setminus(r_k-\d,r_k+\d))}\leq \frac{C}{|\sigma'(r_k)|\d},
\end{align*}
and
\begin{align*}
\left\|\frac{\sigma'}{(\sigma-\nu_k)^2}\right\|_{L^2(\R_+\setminus(r_k-\d,r_k+\d))}^2&\leq\left\|\frac{\sigma'}{\sigma-\nu_k}\right\|_{L^{\infty}(\R_+\setminus(r_k-\d,r_k+\d))}
\left\|\frac{\sigma'}{(\sigma-\nu_k)^3}\right\|_{L^1(\R_+\setminus(r_k-\d,r_k+\d))}\\
&\leq \frac{C}{\d}
\left(\frac{1}{(\sigma(r_k-\d)-\nu_k)^2}+\frac{1}{(v_k-\sigma(r_k+\d))^2}\right)\\
&\leq \frac{C}{\d(\sigma'(r_k)\d)^2}.
\end{align*}
Thus, we obtain
\begin{align*}
\|w\|\le&\|w\|_{L^2(\R_+\setminus(r_k-\d,r_k+\d))}^2+\|w\|_{L^2((r_k-\d,r_k+\d))}^2\\
\leq&\int_{\R_+\setminus(r_k-\d,r_k+\d)}\frac{g^2|\psi|^2}{(\sigma-\nu_k)^2}dr+\frac{C\|w'\|
\|w\|_{L^{\infty}}}{|\b_k\d^{\frac{3}{2}}\sigma'(r_k)|}\\
&+\frac{C\|w'\overline{w}\|_{L^{\infty}(r_k-\d,r_k+\d)}}{|\b_k\d\sigma'(r_k)|}
+\frac{C\|\widetilde{\mathcal{H}}_kw\|\|w\|}{|\b_k\d\sigma'(r_k)|}+2\delta\|w\|_{L^\infty}^2\\
\le& \int_{\R_+\setminus(r_k-\d,r_k+\d)}\frac{g^2|\psi|^2}{(\sigma-\nu_k)^2}dr+C\cE(w).
\end{align*}

It remains to estimate the first term, which is bounded by
\begin{align*}
\int_{\R_+\setminus(r_k-\d,r_k+\d)}\frac{g^2|\psi|^2}{(\sigma-\nu_k)^2}dr\leq 2\int_{\R_+}\frac{g^2|\psi_1|^2}{(\sigma-\nu_k)^2}dr+2\int_{\R_+\setminus(r_k-\d,r_k+\d)}\frac{g^2|\psi_2|^2}{(\sigma-\nu_k)^2}dr.
\end{align*}
Let us first consider the case of $0<r_k\leq 1$. By Lemma \ref{lem:sigma}, we have
\begin{align*}
&|\sigma(r)-\nu_k|\geq C^{-1}|r-r_k||\sigma'(r_k)|\geq C^{-1}|r-r_k|r_k,\quad0<r<r_k+1,\\
&|\sigma(r)-\nu_k|\geq C^{-1},\quad r\geq r_k+1.
\end{align*}
Due to $\psi_1(r_k)=0,$ we get by Hardy's inequality that
\begin{align*}
&\int_{0}^{r_k+1}\frac{g^2|\psi_1|^2}{(\sigma-\nu_k)^2}dr\leq C\int_{0}^{r_k+1}\frac{|\psi_1|^2}{|r-r_k|^2r_k^2}dr\leq \frac{C}{r_k^2}\int_{0}^{r_k+1}{|\partial_r\psi_1|^2}dr
\leq \frac{CA_1}{r_k^2},
\end{align*}
and by Lemma \ref{lem:sigma},
\begin{align*}
&\int_{r_k+1}^{+\infty}\frac{g^2|\psi_1|^2}{(\sigma-\nu_k)^2}dr\leq C\int_{r_k+1}^{+\infty}g^2|\psi_1|^2dr\leq C\int_{r_k+1}^{+\infty}\frac{|\psi_1|^2}{r^2}dr
\leq {CA_1}.
\end{align*}
Thanks to $0<\d<r_k$ and $|\psi_2(r)|\leq |\psi(r_k)|,$ we have
 \begin{align*}
 &\int_{(0,r_k+1)\setminus(r_k-\d,r_k+\d)}\frac{g^2|\psi_2|^2}{(\sigma-\nu_k)^2}dr\leq C\int_{(0,r_k+1)\setminus(r_k-\d,r_k+\d)}\frac{|\psi(r_k)|^2}{|r-r_k|^2r_k^2}dr\leq \frac{C|\psi(r_k)|^2}{r_k^2\d},
 \end{align*}
 and
 \begin{align*}
 &\int_{r_k+1}^{+\infty}\frac{g^2|\psi_2|^2}{(\sigma-v_k)^2}dr\leq C|\psi(r_k)|^2\int_{r_k+1}^{+\infty}g^2(r)dr\leq C|\psi(r_k)|^2.
 \end{align*}
 Therefore, we obtain
\begin{align*}
\int_{\R_+\setminus(r_k-\d,r_k+\d)}\frac{g^2|\psi|^2}{(\sigma-\nu_k)^2}dr\leq& \frac{CA_1}{r_k^2}+\frac{C|\psi(r_k)|^2}{r_k^2\d}\le C\cE(w),
\end{align*}
where we used \eqref{eq:A1-est} and the facts that ${|\sigma'(r_k)|\sim r_k}$, $g(r_k)\geq C^{-1}$ and $|\d\sigma'(r_k)|\leq Cr_k^2$.
\smallskip

Next we consider the case of $r_k\ge 1$. By Lemma \ref{lem:sigma}, we have
\begin{align*}
&|\sigma(r)-\nu_k|\geq C^{-1}|r-r_k||\sigma'(r_k)|\geq C^{-1}|r-r_k|r_k^{-3} ,\quad|r-r_k|<1,\\
&|\sigma(r)-\nu_k|\geq C^{-1}(1+r)^{-4}, \quad|r-r_k|\geq 1/r_k,
\end{align*}
and $g(r)\leq Cg(r_k)$ for  $|r-r_k|<1/r_k$. Thanks to $\psi_1(r_k)=0$ and ${g(r_k-1)^2}r_k^6\leq C$, we get by Hardy's inequality that
\begin{align*}
 &\int_{r_k-1}^{r_k+1}\frac{g^2|\psi_1|^2}{(\sigma-\nu_k)^2}dr\leq C\int_{r_k-1}^{r_k+1}\frac{g^2|\psi_1|^2r_k^6}{|r-r_k|^2}dr\leq {C}\int_{r_k-1}^{r_k+1}{|\partial_r\psi_1|^2}dr
\leq {CA_1},
\end{align*}
and
\begin{align*} &\int_{\R_+\setminus B(r_k,1)}\frac{g^2|\psi_1|^2}{(\sigma-v_k)^2}dr\leq C\int_{\R_+}g^2(r)(1+r)^8|\psi_1|^2dr\leq C\int_{\R_+}\frac{|\psi_1|^2}{r^2}dr
\leq {CA_1},
\end{align*}
where we denote $B(a,b)=(a-b,a+b).$ Since $0<\d<1/r_k$ and $|\psi_2(r)|\leq |\psi(r_k)|$, we have
 \begin{align*}
 &\int_{B(r_k,1/r_k)\setminus B(r_k,\d)}\frac{g^2|\psi_2|^2}{(\sigma-\nu_k)^2}dr\leq C\int_{B(r_k,1/r_k)\setminus B(r_k,\d)}\frac{g^2(r_k)|\psi(r_k)|^2}{|r-r_k|^2\sigma'(r_k)^2}dr\leq \frac{Cg^2(r_k)|\psi(r_k)|^2}{\sigma'(r_k)^2\d},
 \end{align*}
and due to $|\psi_2(r)|\leq (r/r_k)^{\frac{5}{2}}|\psi(r_k)|$, we have
 \begin{align*}
 &\int_{\R_+\setminus B(r_k,1/r_k) }\frac{g^2|\psi_2|^2}{(\sigma-\nu_k)^2}dr\leq C\int_{R_+}(r/r_k)^{5}{g^2(r)|\psi(r_k)|^2}(1+r)^8dr\leq \frac{C|\psi(r_k)|^2}{r_k^5}.
 \end{align*}
Therefore, we obtain
\begin{align*}
\int_{\R_+\setminus B(r_k,\d)}\frac{g^2|\psi|^2}{(\sigma-\nu_k)^2}dr\leq {CA_1}+\frac{Cg^2(r_k)|\psi(r_k)|^2}{\sigma'(r_k)^2\d}+\frac{C|\psi(r_k)|^2}{r_k^5}\le C\cE(w),
\end{align*}
where we used \eqref{eq:A1-est} and the facts that $C^{-1}r_k^{-3}\leq|\sigma'(r_k)|\leq Cr_k^{-3} $ and $|\d\sigma'(r_k)|\leq C$.

This completes the proof of the lemma.
\end{proof}

Now we are in a position to show that for $|\beta_k|\geq\max\big(\f{|k|^3}{r_k^4},|k|^3,r_k^6\big)$,
\begin{align}
\|\widetilde{\mathcal{H}}_kw\|\gtrsim|\beta_k|^{\f13}\|w\|.
\end{align}

By Lemma \ref{lem:w-L2}, we have
\beno
\|w\|\le C\cE(w)=C\big(\cE_1(w)+\cdots+\cE_7(w)\big).
\eeno
In the following, we handle each $\cE_i(w)$. Using the fact that
\begin{align}
&|J(r)|\leq\int_0^{r}|w(s)|ds+\int_r^{+\infty}\left(\dfrac{r}{s}\right)^{\frac{3}{2}}|w(s)|ds\leq Cr\|w\|_{L^{\infty}},\label{eq:J-est1}\\
&|J(r)|\leq\int_0^{+\infty}\left(\dfrac{s}{r}\right)^{\frac{5}{2}}|gw|(s)ds\leq Cr^{-\frac{5}{2}}\min(\|w\|_{L^{\infty}},\|w\|),\label{eq:J-est2}
\end{align}
we deduce that
\begin{align}
&\cE_5(w)\leq \f{C}{r_k}|\psi(r_k)|\|w\|_{L^{\infty}}\leq \cE_4(w)+C\f{1}{\delta}\f{|\psi(r_k)|^2}{r_k^2}\leq \cE_4(w)+C\cE_6(w),\quad 0<r_k\leq1,\label{eq:E5-est1}\\
&\cE_5(w)\leq C\f{1}{r_k^{5/2}}|\psi(r_k)|\|w\|\leq C\cE_7(w)^{\frac{1}{2}}\|w\|,\quad r_k\geq1.\label{eq:E5-est2}
 \end{align}
Using the fact that
\begin{align}
&|\psi(r)|\leq r\int_0^{r}|w(s)|ds+\int_r^{+\infty}\left(\dfrac{r}{s}\right)^{2}(rs)^{\frac{1}{2}}|w(s)|ds\leq Cr^2\|w\|_{L^{\infty}},\label{eq:psi-est1}\\
&|\psi(r)|\leq\int_0^{+\infty}\left(\dfrac{s}{r}\right)^{2}(rs)^{\frac{1}{2}}|gw|(s)ds\leq Cr^{-\frac{3}{2}}\min(\|w\|_{L^{\infty}},\|w\|),\label{eq:psi-est2}
\end{align}
we deduce that
\begin{align}
\cE_7(w)\leq Cr_k^{-8}\|w\|^2,\label{eq:E7-est1}
\end{align}
As $ \d<1,\ |\sigma'(r_k)|<C$, we also have
\begin{align}
\cE_7(w)\leq C\cE_6(w){g(r_k)^{-2}}.\label{eq:E7-est2}
\end{align}
We introduce
\begin{align}
\cF(w)=\delta\|w\|^2_{L^{\infty}}+\delta^2\|w'\|^2+\delta^2\|w\|\|\widetilde{\mathcal{H}}_kw\|+\delta^4\|\widetilde{\mathcal{H}}_kw\|^2.
\end{align}
It is easy to see that
\begin{align}
\cE_4(w)\leq \cF(w),\label{eq:E4-est}
\end{align}
and by \eqref{eq:delta-sigma}, we have
\begin{align}
\cE_1(w)+\cE_2(w)\leq C\cF(w).\label{eq:E12-est}
\end{align}

To proceed, we need the following $L^\infty$ estimate of $w'$ and $\psi$.

\begin{lemma}\label{lem:F}
It holds that
\begin{align*}
\d^3\|w'\|_{L^{\infty}(B(r_k,\d))}^2+\f{1}{(\sigma'(r_k))^2\delta}\|g\psi\|_{L^{\infty}(B(r_k,\d))}^2\leq C\cF(w).
\end{align*}
\end{lemma}

\begin{proof}
Let
\beno
u=\widetilde{\mathcal{H}}_kw,\quad u_1=g\psi,\quad u_2=\Big(\dfrac{k^2-1/4}{r^2}+\dfrac{r^2}{16}-\dfrac{1}{2}\Big)w+i\b_k(\sigma-\nu_k)w.
\eeno
Then we have
\begin{align}\label{eq:u-w}
-w''+u_2-i\b_ku_1=u.
\end{align}
Due to $0<\d\leq\min(\f{r_k}{|k|},\f{1}{r_k})$ {and Lemma \ref{lem:sigma}}, we have
 \begin{align}
 \|u_2\|_{L^{\infty}(B(r_k,\d))}&\leq \Big\|\dfrac{k^2-1/4}{r^2}+\dfrac{r^2}{16}-\dfrac{1}{2}\Big\|_{L^{\infty}(B(r_k,\d))}\|w\|_{L^{\infty}}+\|\b_k(\sigma-\nu_k)\|_{L^{\infty}(B(r_k,\d))}
 \|w\|_{L^{\infty}}\nonumber\\
 &\leq C\Big(\dfrac{k^2}{r_k^2}+{r_k^2}+|\b_k\d\sigma'(r_k)|\Big)\|w\|_{L^{\infty}}\nonumber\\
 &\leq C(\d^{-2}+|\b_k\d\sigma'(r_k)|)\|w\|_{L^{\infty}}\leq C\d^{-2}\|w\|_{L^{\infty}}.\label{eq:u2-est}
\end{align}
By \eqref{eq:psi-est1}, \eqref{eq:J-est1} and \eqref{eq:psi-J}, we get
\begin{align*}
|\partial_r\psi|\leq Cr\|w\|_{L^{\infty}},
\end{align*}
which gives
\begin{align*}
|\partial_r u_1|&\leq g|\partial_r\psi|+|\partial_rg||\psi|\leq g(|\partial_r\psi|+r|\psi|)\\
&\leq Cg(r)^{\frac{1}{2}}(r+r^3)\|w\|_{L^{\infty}}\leq C|\sigma'(r)|\|w\|_{L^{\infty}}.
\end{align*}
In particular, for $r,s\in B(r_k,\d)$,
\begin{align}
|u_1(r)-u_1(s)|\leq C\d|\sigma'(r_k)|\|\omega\|_{L^{\infty}}.\label{eq:u1-est}
\end{align}
Choose $r_*\in(r_k-\d,r_k)$ such that
\begin{align*}
|w'(r_*)|^2+|w'(r_*+\d)|^2\leq \f{1}{\delta}\|w'\|^2,
\end{align*}
which along with \eqref{eq:u-w} gives
\begin{align*}
\left|\int_{r_*}^{r_*+\d}(u_2-i\b_ku_1-u)dr\right|=|w'(r_*+\d)-w'(r_*)|\leq 2\d^{-\frac{1}{2}}\|w'\|,
\end{align*}
from which and \eqref{eq:u2-est}, we infer that
\begin{align*}
 \left|\int_{r_*}^{r_*+\d}u_1dr\right|&\leq |\b_k|^{-1}(\|u_2\|_{L^1(r_*,r_*+\d)} +\|u\|_{L^1(r_*,r_*+\d)}+ 2\d^{-\frac{1}{2}}\|w'\|)\\
 &\leq |\b_k|^{-1}(\d\|u_2\|_{L^{\infty}(r_*,r_*+\d)} +\d^{\frac{1}{2}}\|u\|+ 2\d^{-\frac{1}{2}}\|w'\|)\\
 &\leq |\b_k|^{-1}(C\d^{-1}\|w\|_{L^{\infty}} +\d^{\frac{1}{2}}\|\widetilde{\mathcal{H}}_kw\|+ 2\d^{-\frac{1}{2}}\|w'\|)\\
 &\leq C|\b_k|^{-1}\d^{-\frac{3}{2}}\cF(w)^{\frac{1}{2}}.
\end{align*}
For $s\in B(r_k,\d),$ we get by \eqref{eq:u1-est} that
\begin{align*}& \left|\d u_1(s)-\int_{r_*}^{r_*+\d}u_1dr\right|\leq \int_{r_*}^{r_*+\d}|u_1(s)-u_1(r)|dr\leq C\d^2|\sigma'(r_k)|\|w\|_{L^{\infty}},
\end{align*}
which gives
\begin{align}
 |u_1(s)|\leq C(\d^{\frac{1}{2}}|\sigma'(r_k)|+|\b_k|^{-1}\d^{-\frac{5}{2}})A^{\frac{1}{2}}\leq C\d^{\frac{1}{2}}|\sigma'(r_k)|\cF(w)^{\frac{1}{2}},\label{eq:u1-est2}
\end{align}
that is,
\beno
\|g\psi\|_{L^{\infty}(B(r_k,\d))}=\|u_1\|_{L^{\infty}(B(r_k,\d))}\leq C\d^{\frac{1}{2}}|\sigma'(r_k)|\cF(w)^{\frac{1}{2}}.
\eeno
Using \eqref{eq:u-w}, \eqref{eq:u2-est} and \eqref{eq:u1-est2}, we infer that
 \begin{align*}
 &\|w'\|_{L^{\infty}(B(r_k,\d))}\leq |w'(r_*)|+\|w''\|_{L^{1}(B(r_k,\d))}\\
 &\leq \d^{-\frac{1}{2}}\|w'\|+\|u_2\|_{L^{1}(B(r_k,\d))}+|\b_k| \|u_1\|_{L^{1}(B(r_k,\d))}+\|u\|_{L^{1}(B(r_k,\d))}\\
 &\leq \d^{-\frac{1}{2}}\|w'\|+2\d\|u_2\|_{L^{\infty}(B(r_k,\d))}+2|\b_k|\d \|u_1\|_{L^{\infty}(B(r_k,\d))}+2\d^{\frac{1}{2}}\|u\|\\
 &\leq \d^{-\frac{1}{2}}\|w'\|+C\d^{-1}\|w\|_{L^{\infty}}+C|\b_k|\d^{\frac{3}{2}}|\sigma'(r_k)|\cF(w)^{\frac{1}{2}}+2\d^{\frac{1}{2}}\|u\|\\
 &\leq C\d^{-\frac{3}{2}}\cF(w)^{\frac{1}{2}}.
 \end{align*}
This completes the proof of the lemma.
\end{proof}

Now we infer from Lemma \ref{lem:F} that
\begin{align}
\cE_6(w)\leq C\cF(w),\label{eq:E6-est}
\end{align}
and by \eqref{eq:delta-sigma},
\begin{align}
\cE_3(w)\leq& \f{1}{|\b_k\d\sigma'(r_k)|}\|w'\|_{L^{\infty}(B(r_k,\d))}\|w\|_{L^{\infty}}\nonumber\\
\leq& C\d^2\|w'\|_{L^{\infty}(B(r_k,\d))}\|w\|_{L^{\infty}}\leq C\cF(w).\label{eq:E3-est}
\end{align}

If $0<r_k\leq1,$ we deduce from \eqref{eq:E12-est}, \eqref{eq:E3-est}, \eqref{eq:E4-est}, \eqref{eq:E5-est1}, \eqref{eq:E6-est}, \eqref{eq:E7-est2} that
\begin{align*}
\|w\|^2\leq C\cF(w).
\end{align*}
If $r_k\geq1,$ we similarly have
\begin{align*}
\|w\|^2\leq C\big(\cF(w)+\cE_7(w)^{\frac{1}{2}}\|w\|+\cE_7(w)\big)\le {C_0}\big(\cF(w)+\cE_7(w)\big).
\end{align*}
Now if $\|w\|^2< 2{C_0}\cE_7(w),$ we get by \eqref{eq:E7-est1} that
\begin{align*}
\|w\|^2< 2Cr_k^{-8}\|w\|_{L^{2}}^2,
\end{align*}
which implies that $r_k\leq C$, thus, ${g(r_k)^{-2}}\leq C$. Hence,
\begin{align*}
\cE_7(w)\leq C\cE_6(w)\leq CA \Longrightarrow\|w\|^2\leq C\cF(w).
\end{align*}
While, if $\|w\|^2\geq 2C_0\cE_7(w)$, we have
\begin{align*}
\|w\|^2\leq 2C\cF(w).
\end{align*}

Thanks to $\|w\|^2_{L^{\infty}}\leq \|w'\|\|w\|$ and $\|w'\|^2\leq\|w\|\|\widetilde{\mathcal{H}}_kw\|$, we have
\begin{align*}
\cF(w)\leq &\delta \|w' \|\|w\|+2\delta^2\|w\|\|\widetilde{\mathcal{H}}_kw\|+\delta^4\|\widetilde{\mathcal{H}}_kw\|^2\\
\leq &\|w\|^{\f32}(\delta^2\|\widetilde{\mathcal{H}}_kw\|)^{\f12}+2\|w\|_{L^2}(\delta^2\|\widetilde{\mathcal{H}}_kw\|)
+((\delta^2\|\widetilde{\mathcal{H}}_kw\|))^2,
\end{align*}
which along with $\|w\|^2\le C\cF(w)$ implies that
\begin{align*}
\|w\|\le C\delta^2\|\widetilde{\mathcal{H}}_kw\|.
\end{align*}
Due to the choice of $\d$, we obtain
\begin{align*}
\|\widetilde{\mathcal{H}}_kw\|\gtrsim |\beta_k|^{\f13}\|w\|.
\end{align*}

\section{Spectral lower bound}

Recall that
\begin{align}
\widetilde{\mathcal{H}}_{\alpha,k ,0}=-\partial^2_r+\f{k^2-\f14}{r^2}+\f{r^2}{16}-\f12+i\beta_k\sigma(r)-i\beta_k g\widetilde{\mathcal{K}}_k[g\cdot].\nonumber
\end{align}
We know that for $|k|=1$, $\widetilde{\mathcal{H}}_{\alpha,k,0}$ in $\{r^{\f32}g(r)\}^{\bot}\cap L^2(\mathbb{R}_+, dr)$ is isometric with $T\widetilde{\mathcal{H}}_{\alpha,k,0}T^{-1}=-\partial_r^2+\f{3}{4r^2}+\f{r^2}{16}-\f12+f(r)+i\beta_k\sigma(r)$ in $ L^2(\mathbb{R}_+, dr)$. Hence, we just consider $\widetilde{\mathcal{H}}_{\alpha,k}$ in the form
\begin{align*}
\widetilde{\mathcal{H}}_{\alpha,k}=
\left\{
\begin{aligned}
&-\partial_r^2+\f{3}{4r^2}+\f{r^2}{16}-\f12+f(r)+i\beta_k\sigma(r),\quad|k|=1,\\
&-\partial_r^2+\f{k^2-\f14}{r^2}+\f{r^2}{16}-\f12+i\beta_k\sigma(r)-i\beta_k g\tilde{\mathcal{K}}_k[g\cdot],\quad |k|\geq2.
\end{aligned}
\right.
\end{align*}
Notice that
\beno
\sigma(L-\al\Lambda|_{(\ker \Lambda)^{\bot}})=\bigcup\limits_{k\in \mathbb{Z}\setminus\{0\}}\sigma(\widetilde{\cH}_{\al,k}).
\eeno
Then we define
\begin{align*}
\Sigma(\alpha,k)=\inf{\rm Re}\,\sigma(\widetilde{\mathcal{H}}_{\alpha,k}), \quad\Sigma(\alpha)=\inf\limits_{k\in\mathbb{Z}\setminus\{0\}}\Sigma(\alpha,k).
\end{align*}

Our main result is the following spectral lower bound.
\begin{theorem}\label{thm:spectral}
For any $|k|\ge 1$, we have
\begin{align}
\Sigma(\alpha,k)\geq C^{-1}|\beta_k|^{\frac{1}{2}},\quad \Sigma(\alpha)\geq C^{-1}|\alpha|^{\frac{1}{2}}.\nonumber
\end{align}
\end{theorem}

Motivated by \cite{GGN}, we will use the complex deformation method.

\subsection{Complex deformation}
We introduce the group of dilations
\beno
(U_{\theta}\omega)(r)=e^{\theta/2}\omega(e^{\theta}r),
\eeno
which are unitary operators for $\theta\in\R$. We consider
\begin{align}
\widetilde{\mathcal{H}}_{\alpha,k}^{(\theta)}=U_{\theta}\widetilde{\mathcal{H}}_{\alpha,k}U_{\theta}^{-1}.
\end{align}
Then we have
\begin{align*}
\widetilde{\mathcal{H}}_{\alpha,k}^{(\theta)}=
\left\{
\begin{aligned}
&-e^{-2\theta}\partial_r^2w+\left(\dfrac{3e^{-2\theta}}{4r^2}+\dfrac{r^2e^{2\theta}}{16}
-\dfrac{1}{2}+f(re^{\theta})\right)+i\b_k\sigma(re^{\theta}),\quad|k|=1,\\
&-e^{-2\theta}\partial_r^2+\left(\dfrac{k^2-1/4}{4r^2e^{2\theta}}+\dfrac{r^2e^{2\theta}}{16}
-\dfrac{1}{2}\right)+i\beta_k(\sigma(re^{\theta})-e^{2\theta}
g(re^{\theta})\widetilde{\mathcal{K}}_k[g(re^{\theta})\cdot]),\quad |k|\geq2.
\end{aligned}
\right.
\end{align*}
Now we consider the analytic continuation of $\widetilde{\mathcal{H}}_{\alpha,k}^{(\theta)}$.
For this, we first consider the analytic continuation of the functions $f,\sigma, g$. Let
 \begin{align*}
&F_0(z)=e^z-z-1,\quad F_1(z)=(1-e^{-z})/z,\\
&F_2(z)=e^{z/2},\quad F_3(z)=\left(\dfrac{z^2}{F_0(z)}-6+4z\right)\dfrac{z/2}{F_0(z)}.
 \end{align*}
 Then $F_0,\ F_1,\ F_2$ are holomorphic in $ \mathbb{C}$ (0 is a removable singularity of $F_1$) and $F_3$ is meromorphic in $ \mathbb{C}$,
 and we have
 \begin{align*}
 f(r)=F_3(r^2/4),\quad \sigma(r)=F_1(r^2/4),\quad g(r)=F_2(r^2/4).
 \end{align*}

The poles of $F_3$ are the zeros of $F_0.$ If $F_0(z)=0,\ z=x+iy,\ x,y\in\R,\ x>0,$ then
 \begin{align*}
 &e^{2x}=|e^z|=|1+z|^2=(1+x)^2+y^2, \\
&y^2=e^{2x}-(1+x)^2>1+2x+(2x)^2/2-(1+x)^2=x^2\Longrightarrow|y|>x,
 \end{align*}
hence, $F_3(z)$ is holomorphic in a neighbourhood of $\Gamma$, which is defined as
\begin{align*}
\Gamma=\big\{x+iy|x>0,-x\leq y\leq x\big\}=\big\{re^{i\theta}|r>0,-\pi/4\leq \theta\leq \pi/4\big\}.
\end{align*}
Let $F_4(z)=F_3(z)-\f{8}{z}$. As $\lim\limits_{z\to0}zF_3(z)=8$, $F_4(z)$ is holomorphic in a neighbourhood of $\Gamma\cup\{0\}$.
We have
\begin{align*}
|F_0(z)|\geq |e^z|-|z|-1\geq e^{|z|/2}-|z|-1, \quad z\in\Gamma,
\end{align*}
and
\begin{align*}
\lim_{z\to\infty, z\in \Gamma} \f{z^2}{F_0(z)}=0,\quad \lim_{z\to\infty, z\in \Gamma} F_3(z)=0,\quad \lim_{z\to\infty, z\in \Gamma} F_4(z)=0,
\end{align*}
thus $|F_4(z)|\leq C$ in $\Gamma$. We also have
\begin{align*}
F_2(z)\leq C(1+|z|)^{-1}.
\end{align*}

Now we rewrite $\widetilde{\mathcal{H}}_{\alpha,k}^{(\theta)}\omega$ as follows, {for $|k|=1$,}
\begin{align*}
\widetilde{\mathcal{H}}_{\alpha,{k}}^{(\theta)}\omega=&-e^{-2\theta}\partial_r^2w
+\left(\dfrac{35e^{-2\theta}}{4r^2}+\dfrac{r^2e^{2\theta}}{16}-\dfrac{1}{2}+F_4(\f{r^2e^{2\theta}}{4})\right)w
+i\b_kF_1(\f{r^2e^{2\theta}}{4})w,
\end{align*}
and for $|k|\ge 2$,
\begin{align*}
\widetilde{\mathcal{H}}_{\alpha,k}^{(\theta)}w=&-e^{-2\theta}\partial_r^2w
+\left(\dfrac{k^2-1/4}{{r^2e^{2\theta}}}+\dfrac{r^2e^{2\theta}}{16}-\dfrac{1}{2}\right)w+i\beta_k(F_1(\f{r^2e^{2\theta}}{4})w\\
&-e^{2\theta}F_2(\f{r^2e^{2\theta}}{4})\widetilde{\mathcal{K}}_k[F_2(\f{r^2e^{2\theta}}{4})w]).
\end{align*}

Thanks to the properties of $F_i(z)(i=0,1,\cdots,4)$ which are shown above,
$\{\widetilde{\mathcal{H}}_{\alpha,k}^{(\theta)}\}$ are defined as an analytic family of type (A) in the strip $\Gamma_1=\big\{\theta\in \mathbb{C}\big||\Im \theta|<\f{\pi}{8}\big\}$ with common domain ${{D}}=\big\{\omega\in H^2(\R_+)|\omega/r^2,r^2\omega\in L^2(\R_+)\big\}.$
In particular, the spectrum of $\widetilde{\mathcal{H}}_{\alpha,k}^{(\theta)} $ is always discrete and depends holomorphically on $\theta$.
Since the eigenvalues of $\widetilde{\mathcal{H}}_{\alpha,k}^{(\theta)}$ are constant for $\theta\in\R$,
they are also constant for $\theta\in\Gamma_1.$

Now we have
\begin{align}
\Sigma(\alpha,k)=\inf{\rm Re}\sigma(\widetilde{\mathcal{H}}_{\alpha,k}^{(\theta)})\geq\inf_{w\in \underline{{D}},\|w\|=1}{\rm Re}\langle \widetilde{\mathcal{H}}_{\alpha,k}^{(\theta)}w,w\rangle.\label{eq:spectral}
\end{align}

\subsection{Proof of Theorem \ref{thm:spectral}}

We need the following lemma.
\begin{lemma}\label{lem:F1-est}
For $r>0,0<\theta<\f{\pi}{4}$, we have
\begin{align}
-{\rm Im} F_1(r^{i\theta})\geq C^{-1}\sin\theta\min\big(r,\f{1}{r}\big).\nonumber
\end{align}
\end{lemma}
\begin{proof}
Thanks to
\begin{align*}
F_1(r^{i\theta})=\f{1-e^{-re^{i\theta}}}{re^{i\theta}}=\f{e^{-i\theta}-e^{-re^{i\theta}-i\theta}}{r},
\end{align*}
we have
\begin{align*}
-{\rm Im}F_1(r^{i\theta})=\f{F_5(r,\theta)}{r},
\end{align*}
where
\begin{align*}
F_5(r,\theta)&=-{\rm Im}(e^{-i\theta}-e^{-re^{i\theta}-i\theta})\\
&=\sin\theta-e^{-r\cos{\theta}}\sin(r\sin{\theta}+\theta).
\end{align*}
Using the inequality
\begin{align*}
|\sin(r\sin{\theta}+\theta)|& \leq |\sin(r\sin{\theta})\cos\theta|+|\cos(r\sin{\theta})\sin\theta|\\
&\leq r\sin{\theta}\cos\theta+\sin\theta,
\end{align*}
we get
\begin{align}
F_5(r,\theta)\geq \sin\theta(1-e^{-r\cos{\theta}}(1+r\cos{\theta})).\label{eq:F5-est1}
\end{align}
This shows that
\begin{align*}
F_5(r,\theta)\geq C^{-1}\min((r\cos{\theta})^2,1)\sin\theta\geq C^{-1}\min(r^2,1)\sin\theta,
 \end{align*}
thus,
\begin{align*}
 -{\rm Im F_1}(r^{i\theta})=\f{F_5(r,\theta)}{r}\geq C^{-1}\min(r,\f{1}{r})\sin\theta.
\end{align*}
This completes the proof.
\end{proof}

Now we are in a position to prove Theorem \ref{thm:spectral}.\smallskip

Let us first consider the case of $|k|=1.$ It follows from Lemma \ref{lem:coer-A1} that
\begin{align*}
{\rm Re}\langle\widetilde{\mathcal{H}}_{\alpha,k}w,w\rangle\geq \f{\|w\|^2}{2},
\end{align*}
which gives
\begin{align}
\Sigma(\alpha,k)\geq 1/2.\label{eq:S1}
\end{align}

For $ \theta\in(-\pi/8,\pi/8),\ \text{sgn} \theta=\text{sgn} \b_k$, we have
 \begin{align*}
 {\rm Im} F_1(\f{r^2e^{2i\theta}}{4})=\text{sgn} \theta{ \rm Im} F_1(\f{r^2e^{2i|\theta|}}{4}),
 \end{align*}
from which and Lemma \ref{lem:F1-est}, we infer that
\begin{align*}
{\rm Re}\langle \widetilde{\mathcal{H}}_{\alpha, {k}}^{(i\theta)}w,w\rangle=&\int_{\R_+}{\rm Re}\left(\dfrac{35e^{-2i\theta}}{4r^2}
+\dfrac{r^2e^{2i\theta}}{16}-\dfrac{1}{2}+F_4(\f{r^2e^{2i\theta}}{4})+i\b_kF_1(\f{r^2e^{2i\theta}}{4})\right)|w|^2dr+\cos{2\theta}\|\partial_rw\|^2\\
 \geq&\int_{\R_+}\left(\dfrac{35\cos (2\theta)}{4r^2}+\dfrac{r^2\cos (2\theta)}{16}
-\dfrac{1}{2}-C-|\beta_k|{\rm Im} F_1(\f{r^2e^{2i|\theta|}}{4})\right)|w|^2dr\\
\geq&\int_{\R_+}\left(C^{-1}(\f{1}{r^{2}}+r^2)-C+|\beta_k|C^{-1}\sin|\theta|\min(r^2,\f{1}{r^2})\right)|w|^2dr\\ \geq&\int_{\R_+}\left(C^{-1}|\beta_k\sin\theta|^{\frac{1}{2}}-C\right)|w|^2dr=\left(C^{-1}|\b_k\sin\theta|^{\frac{1}{2}}-C\right)\|w\|^2,
\end{align*}
which shows that
\begin{align}
\Sigma(\alpha,k)\geq C^{-1}|\beta_k\sin\theta|^{\frac{1}{2}}-C\ge C^{-1}|\beta_k|^{\frac{1}{2}}-C,\nonumber
\end{align}
 if we take $\theta=(\text{sgn} \beta_k)\f{\pi}{12}$. Then by \eqref{eq:S1}, we get
\begin{align*}
\Sigma(\alpha,k)\geq \max(C^{-1}|\beta_k|^{\frac{1}{2}}-C,1/2)\geq C^{-1}|\beta_k|^{\frac{1}{2}}.
\end{align*}

Next we consider the case of $|k|\geq 2.$ We still assume $\theta\in(-\f{\pi}{8},\f{\pi}{8}),\ \text{sgn}\theta=\text{sgn} \beta_k$. Then we have
\begin{align*}
{\rm Re}\langle \widetilde{\mathcal{H}}_{\alpha,k}^{(i\theta)}w,w\rangle=&\cos{2\theta}\|\partial_rw\|_{L^2}^2+\int_{\R_+}{\rm Re}\left(\dfrac{k^2-1/4}{r^2e^{-2i\theta}}
+\dfrac{r^2e^{2i\theta}}{16}-\dfrac{1}{2}+i\beta_kF_1(\f{r^2e^{2i\theta}}{4})\right)|w|^2dr\\
&+\int_{\R_+}\int_{\R_+}K_k(r,s){\rm Re}(-i\beta_ke^{2i\theta}F_2(\f{r^2e^{2i\theta}}{4})F_2(s^2e^{2i\theta}/4))\overline{w(s)}w(r)drds,
\end{align*}
here $K_k(r,s)=\dfrac{1}{2|k|}\min\left(\dfrac{r}{s},\dfrac{s}{r}\right)^{|k|}(rs)^{\frac{1}{2}}.$ Notice that
\begin{align*}
&{\rm Re}\Big(-i\b_ke^{2i\theta}F_2(\f{r^2e^{2i\theta}}{4})F_2(\f{s^2e^{2i\theta}}{4})\Big)={\rm Re}\big(-i\b_ke^{2i\theta}e^{-(r^2+s^2)\f{e^{2i\theta}}{8}}\big)\\
&=\beta_ke^{-(r^2+s^2)\f{\cos(2\theta)}{8}}\sin(2\theta-\f{(r^2+s^2)\sin(2\theta)}{8})\\
&=\frac{\beta_k}{\sin(2\theta)}
e^{-(r^2+s^2)\f{\cos(2\theta)}{8}}\Big(\sin(2\theta-\f{r^2\sin(2\theta)}{8})\sin(2\theta-\f{s^2\sin(2\theta)}{8})\\
&\quad-\sin(\f{r^2\sin(2\theta)}{8})\sin(\f{s^2\sin(2\theta)}{8})\Big)\\
&=\frac{|\beta_k|}{|\sin(2\theta)|}(g_2(r)g_2(s)-g_3(r)g_3(s)),
\end{align*}
where
\begin{align*}
g_2(r)=e^{-r^2\f{\cos(2\theta)}{8}}\sin\Big(2\theta-\f{r^2\sin(2\theta)}{8}\Big),\quad g_3(r)=e^{-r^2\f{\cos(2\theta)}{8}}\sin\Big(\f{r^2\sin(2\theta)}{8}\Big),
\end{align*}
and here we used the fact that
\begin{align*}
\sin(a-b-c)\sin a=\sin(a-b)\sin (a-c)-\sin b \sin c.
\end{align*}
Thus, we obtain
 \begin{align*}
{\rm Re}\langle \widetilde{\mathcal{H}}_{\alpha,k}^{(i\theta)}w,w\rangle=&\cos{2\theta}\|\partial_rw\|_{L^2}^2
+\int_{\R_+}\left(\dfrac{k^2-1/4}{r^2}\cos(2\theta)+\dfrac{r^2}{16}\cos(2\theta)-\dfrac{1}{2}-|\b_k|\Im F_1(\f{r^2e^{2i|\theta|}}{4})\right)|w|^2dr\\
&+\frac{|\b_k|}{|\sin(2\theta)|}\big(\langle\widetilde{\mathcal{K}}_k[g_2w],g_2w\rangle-\langle\widetilde{\mathcal{K}}_k[g_3w],g_3w\rangle\big).
\end{align*}
By the proof of Lemma \ref{lem:HktoLk}, we know that
\begin{align*}
\langle\widetilde{\mathcal{K}}_k[g_2w],g_2w\rangle\geq 0.
\end{align*}
Due to $0<K_k(r,s)\leq K_2(r,s)$, we have
\begin{align*}
\langle \widetilde{\mathcal{K}}_k[g_3w],g_3w\rangle&\leq\int_0^{+\infty}\int_0^{+\infty}
K_k(r,s)|g_3w(s)||g_3w(r)|dsdr\\
&\leq\frac{1}{2}\int_0^{+\infty}\int_0^{+\infty}K_2(r,s)\left(\left(\frac{r}{s}\right)^{\frac{1}{2}}g_3(r)^2|w(s)|^2
+\left(\frac{s}{r}\right)^{\frac{1}{2}}g_3(s)^2|w(r)|^2\right)dsdr\\
&=\int_0^{+\infty}\widetilde{\mathcal{K}}_2[r^{\frac{1}{2}}g_3^2](s)\frac{|w(s)|^2}{s^{\frac{1}{2}}}ds.
\end{align*}
Therefore, we obtain
\begin{align*}
{\rm Re}\langle \widetilde{\mathcal{H}}_{\alpha,k}^{(i\theta)}w,w\rangle\geq&\int_{\R_+}\left(\dfrac{k^2-1/4}{r^2}\cos(2\theta)+\dfrac{r^2}{16}\cos(2\theta)-\dfrac{1}{2}\right)|w|^2dr\\
&+|\beta_k|\int_{\R_+}\left(-{\rm Im}F_1(\f{r^2e^{2i|\theta|}}{4})-\f{\widetilde{\mathcal{K}}_2[r^{\frac{1}{2}}g_3^2]}{|\sin(2\theta)|r^{\frac{1}{2}}}\right)|w|^2dr.
\end{align*}

Due to $0\leq g^2_3(r)\leq e^{-r^2\f{\cos(2\theta)}{4}}|\f{r^2\sin(2\theta)}{8}|^2,$ we have
\begin{align*}
 \frac{\widetilde{\mathcal{K}}_2[r^{\frac{1}{2}}g_3^2]}{|\sin(2\theta)|r^{\frac{1}{2}}}\leq&\frac{|\sin(2\theta)|}{64r^{\frac{1}{2}}}
 \widetilde{\mathcal{K}}_2[r^{\frac{9}{2}}e^{-r^2\f{\cos(2\theta)}{4}}]\\
=&\frac{|\sin(2\theta)|}{64r^{\frac{1}{2}}}\frac{1}{4}\left(\int_0^r\left(\frac{s}{r}\right)^2(rs)^{\frac{1}{2}}
s^{\frac{9}{2}}e^{-s^2\f{\cos(2\theta)}{4}}ds+\int_r^{+\infty}\left(\frac{r}{s}\right)^2(rs)^{\frac{1}{2}}
s^{\frac{9}{2}}e^{-s^2\f{\cos(2\theta)}{4}}ds\right)\\
=&\frac{|\sin(2\theta)|}{256r^{2}}\left(\int_0^rs^7
e^{-s^2\cos(2\theta)/4}ds+r^4\int_r^{+\infty}
s^{3}e^{-s^2\cos(2\theta)/4}ds\right)\\
=&\frac{|\sin(2\theta)|}{2r^{2}|\cos(2\theta)|^4}\left(\int_0^a\rho^3
e^{-\rho}d\rho+r^4\left(\frac{\cos(2\theta)}{4}\right)^2\int_a^{+\infty}
\rho e^{-\rho}d\rho\right)\\
=&\frac{|\sin(2\theta)|}{2r^{2}|\cos(2\theta)|^4}\Big(6-(6+6a+3a^2+a^3)
e^{-a}+a^2(1+a)e^{-a}\Big)\\=&\frac{|\sin(2\theta)|\left(3-(3+3a+a^2)
e^{-a}\right)}{r^{2}|\cos(2\theta)|^4}.
\end{align*}
here $a=\f{r^2\cos(2\theta)}{4}$ and we used the change of variable $\rho=\f{s^2\cos(2\theta)}{4}.$  On the other hand, thanks to ${\rm -Im}F_1(\f{r^2e^{2i|\theta|}}{4})=\f{F_5(\f{r^2}{4},2|\theta|)}{r^2/4}$, we get by \eqref{eq:F5-est1} that
\begin{align*}
-{\rm Im}F_1(\f{r^2e^{2i|\theta|}}{4})\geq& |\sin(2\theta)|\f{1-e^{-r^2\f{\cos{ 2\theta}}{4}}(1+r^2\f{\cos{2\theta}}{4})}{r^2/4}\\
=&4|\sin(2\theta)|\f{1-e^{-a}(1+a)}{r^2}.
 \end{align*}
Then we have
\begin{align*}
\frac{\widetilde{\mathcal{K}}_2[r^{\frac{1}{2}}g_3^2]}
{|\sin(2\theta)|r^{\frac{1}{2}}}\leq \frac{-3{\rm Im}F_1(\f{r^2e^{2i|\theta|}}{4})}{4|\cos(2\theta)|^4}.
\end{align*}
Now we take $ \theta=(\text{sgn} \b_k)\f{\pi}{24}$, then we have $|\cos(2\theta)|^4> 3/4 $, and
\begin{align*}
\dfrac{k^2-1/4}{r^2}\cos(2\theta)+\dfrac{r^2}{16}\cos(2\theta)-\dfrac{1}{2}\geq C^{-1}\big(\f{1}{r^2}+r^2\big).
\end{align*}
Then we conclude that
\begin{align*}
{\rm Re}\langle\widetilde{\mathcal{H}}_{\alpha,k}^{(i\theta)}w,w\rangle\geq& C^{-1}\int_{\R_+}\left(\dfrac{1}{r^2}+{r^2}\right)|w|^2dr\\
&+|\beta_k|\int_{\R_+}\left(-{\rm Im}F_1(\f{r^2e^{2i|\theta|}}{4})+\frac{3 {\rm Im}F_1(\f{r^2e^{2i|\theta|}}{4})}{4|\cos(2\theta)|^4}\right)|w|^2dr\\
\geq& C^{-1}\int_{\R_+}\left(\dfrac{1}{r^2}+{r^2}-|\beta_k|{\rm Im}F_1(\f{r^2e^{2i|\theta|}}{4})\right)|w|^2dr\\
\geq& C^{-1}\int_{\R_+}\left(\dfrac{1}{r^2}+{r^2}+|\beta_k||\sin(2\theta)|\min(r^2,\f{1}{r^2})\right)|w|^2dr \\
\geq& C^{-1}\int_{\R_+}|\beta_k|^{\frac{1}{2}}|w|^2dr= C^{-1}|\beta_k|^{\frac{1}{2}}\|w\|^2,
\end{align*}
which shows that for $|k|\ge 2$,
\beno
\Sigma(\alpha,k)\geq C^{-1}|\b_k|^{\frac{1}{2}}.
\eeno

\section{Appendix}

In this appendix, let us present some properties of the function $\sigma(r)=\f{1-e^{-r^2/4}}{r^2/4}$.

\begin{lemma}\label{lem:sigma}
It holds that
\begin{itemize}

\item[1.] for any $r_0>0$,
\begin{align*}
&|\sigma'(r)|\sim|\sigma'(r_0)|,\quad\f{r_0}{2}\leq r\leq 2r_0,\\
&|\sigma(r)-\sigma(r_0)|\gtrsim |r-r_0||\sigma'(r_0)|,\quad0<r\leq2r_0;
\end{align*}

\item[2.] for $0<r_0<1$ and ${r_0/2}< r\leq 2r_0+1$,
\begin{align*}
&|\sigma'(r)|\gtrsim|\sigma'(r_0)|,\quad |\sigma(r)-\sigma(r_0)|\gtrsim |r-r_0||\sigma'(r_0)|;
\end{align*}

\item[3.] for $r_0\ge 1$ and $|r- r_0|\geq\f{1}{r_0}$,
\begin{align*}
|\sigma(r)-\sigma(r_0)|\gtrsim \f{1}{(1+r)^4}.
\end{align*}
\end{itemize}
Here $a \sim b$ means $ca\le a \le c^{-1}b$ and $a \gtrsim b$ means $a\le Cb$, where $c$ and $C$ are constants
independent of $r_0$.
\end{lemma}

\begin{proof}
Let us prove the first property. Thanks to $\sigma'(r)=\f{2}{r}(e^{-\f{r^2}{4}}-\f{1-e^{-\f{r^2}{4}}}{r^2/4}),$ we have
\begin{align}
|\sigma'(r)|\sim \min\big(r,\f{1}{r^3}\big),\label{eq:sigma-est1}
\end{align}
which shows that if $r_0\geq1$, i.e $r\geq\f12$, then
\begin{align*}
|\sigma'(r)|\sim \f{1}{r^3}\sim\f{1}{r_0^3}\sim|\sigma'(r_0)|,
\end{align*}
and if $r_0\leq1$, i.e $r\leq2$, then
\begin{align*}
|\sigma'(r)|\sim r\sim r_0\sim |\sigma'(r_0)|.
\end{align*}
Thus, $|\sigma'(r)|\sim|\sigma'(r_0)|$ for $\f{r_0}{2}\leq r\leq 2r_0$. If $\f{r_0}{2}\leq r\leq 2r_0$, $|\sigma'(\theta r+(1-\theta)r_0)|\sim|\sigma'(r_0)|(0\leq\theta\leq1)$. Thus, for some $\theta\in(0,1)$,
\begin{align*}
|\sigma(r)-\sigma(r_0)|=|r-r_0||\sigma'(\theta r+(1-\theta)r_0)|\sim|r-r_0||\sigma'(r_0)|.
\end{align*}
While, if $0<r\leq\f{r_0}{2}$, we get by $\sigma'(r)<0$ that
\begin{align*}
|\sigma(r)-\sigma(r_0)|\geq \sigma(\f{r_0}{2})-\sigma(r_0)\sim \f{r_0}{2}|\sigma'(r_0)|\sim|r-r_0||\sigma'(r_0)|.
\end{align*}

The second property could be proved similarly.

Now we prove the third property. If $r_0\geq1$ and $r\geq r_0+\f{1}{r_0}$, we get by \eqref{eq:sigma-est1} that
\begin{align*}
|\sigma(r)-\sigma(r_0)|\geq\sigma(r_0)-\sigma(r_0+\f{1}{r_0})\gtrsim \f{1}{r_0}|\sigma'(r_0)|\sim\f{1}{r_0^4}\gtrsim \f{1}{(1+r)^4},
\end{align*}
and if $r_0\geq1$, $0<r\leq r_0-1$, then
\begin{align*}
|\sigma(r)-\sigma(r_0)|\geq \sigma(r)-\sigma(r+1)\gtrsim|\sigma'(r+1)|\gtrsim\f{1}{(1+r)^4},
\end{align*}
and if $r_0\geq1$, $r_0-1<r\leq r_0-\f{1}{r_0}$, then
\begin{align*}
|\sigma(r)-\sigma(r_0)|\geq\sigma(r_0-\f{1}{r_0})-\sigma(r_0)\gtrsim \f{1}{r_0}|\sigma'(r_0)|\sim\f{1}{r_0^4}\gtrsim \f{1}{(1+r)^4}.
\end{align*}
This shows the third property of $\sigma(r)$.
\end{proof}

\begin{lemma}\label{lem:beta-med}
Let $0<\nu_1<1$ and $\sigma(r_1)=\nu_1$. There exist constants {$c_i\sim1(i=1,\cdots,4)$}, such that for any $r>0$, we have
\begin{itemize}

\item[1.] if $r_1\leq1$ and $1\leq|\beta_1|\leq\f{1}{r_1^4}$, then
\beno
 c_1\f{1}{r^2}+c_2|\beta_1|(\nu_1-\sigma(r))\geq |\beta_1|^{\f12};
\eeno

\item[2.] if  $r_1\geq1$ and $1\leq|\beta_1|\leq r_1^4$, then
\begin{align*}
c_3(1+r^2)+c_4|\beta_1|(\sigma(r)-\nu_1)\geq |\beta_1|^{\f12};
\end{align*}

\item[3.] if  $r_1\geq 1$ and $r_1^4\leq|\beta_1|\leq r_1^6$, then
\begin{align*}
c_3(1+r^2)+c_4{r_1^4}(\sigma(r)-\nu_1)\geq |\beta_1|^{\f13}.
\end{align*}

\end{itemize}

\end{lemma}
\begin{proof}
We consider the first case.
Let $F(r)=c_1\f{1}{r^2}+c_2|\beta_1|(\sigma(r_1)-\sigma(r)),\ r_0=|\beta_1|^{-\f14}$. Then we have $r_1\leq r_0\leq 1$ and
\begin{align*}
F'(r)=-c_2|\beta_1|\sigma'(r)-\f{2c_1}{r^3}.
\end{align*}
If we choose $2c_1=-c_2|\beta_1|\sigma'(r_0)r_0^3,$ due to $-\sigma'(r)\sim\min(r,\f{1}{r^3})$,  we have $ c_1\sim c_2$ and $F'(r_0)=0.$ As $-(r^3\sigma'(r))'=r^3g^2(r)>0$, we conclude that
\begin{align*}
F'(r)<0\ \ \text{for}\ \ \ 0<r<r_0,\ \ F'(r)>0\ \ \text{for}\ \ \ r>r_0,
\end{align*}
which imply that
\begin{align*}
\min\limits_{r>0}F(r)= F\big(r_0\big)\geq \f{c_1}{r_0^2}=c_1|\beta_1|^{\f12}.
\end{align*}
That is, for $c_1= 1$ we have,
\beno
F(r)\geq |\beta_1|^\f12.
\eeno

Next we consider the second case. Let $G(r)=c_3(1+r^2)+c_4|\beta_1|(\sigma(r)-\sigma(r_1))$. Then
\begin{align*}
G(r)\geq c_3(1+r^2)+\frac{c_4|\beta_1|}{C(1+r^2)}-\frac{Cc_4|\beta_1|}{r_1^2}
\geq C^{-1}(c_3c_4|\beta_1|)^{\frac{1}{2}}-\frac{Cc_4|\beta_1|}{|\beta_1|^{\frac{1}{2}}}.
\end{align*}
We can choose constants $c_3,c_4>0$ such that $C^{-1}(c_3c_4)^{\frac{1}{2}}-Cc_4=1 $. Then
\begin{align*}
G(r)\ge  |\beta_1|^{\f12}.
\end{align*}

Finally, we prove the third case. Let $H(r)=c_3{(1+r^2)}+c_4r_1^4(\sigma(r)-\nu_1)$. Then we have
\begin{align*}
H(r)\geq c_3(1+r^2)+\frac{c_4r_1^4}{C(1+r^2)}-\frac{Cc_4r_1^4}{r_1^2}
\geq C^{-1}(c_3c_4)^{\frac{1}{2}}r_1^2-{Cc_4r_1^2}=r_1^2\ge |\beta_1|^\f13.
\end{align*}

The proof is finished.
\end{proof}

Similar to Lemma \ref{lem:beta-med}, we have

\begin{lemma}\label{lem:beta-high}
Let $0<\nu_k<1$ and $\sigma(r_k)=\nu_k$. There exist constants {$c_i\sim1(i=1,\cdots,4)$} such that for any $r>0$, we have
\begin{itemize}

\item[1.] if  $r_k\leq 1$ and $|k|^3\leq|\beta_k|\leq\f{|k|^3}{r_k^4}$, then
\beno
 c_1\f{{k^2}}{r^2}+c_2|\beta_k|(\nu_k-\sigma(r))\geq |\beta_k|^{\f12};
\eeno

\item[2.] if  $r_k\geq \sqrt{k}$ and $|k|^3\leq|\beta_k|\leq r_k^4$, then
\begin{align*}
c_3(1+r^2)+c_4|\beta_k|({\sigma(r)/2}-\nu_k)\geq |\beta_k|^{\f12};
\end{align*}

\item[3.] if  $r_k\geq \sqrt{k}$ and $r_k^4\leq|\beta_k|\leq r_k^6$, then
\begin{align*}
c_3(1+r^2)+c_4{r_k^4}({\sigma(r)/2}-\nu_k)\geq |\beta_k|^{\f13}.
\end{align*}

\end{itemize}
\end{lemma}

\section*{Acknowledgement}

Z. Zhang is partially supported by NSF of China under Grant 11425103.


\begin{thebibliography}{99}

\bibitem{BM1} J. Bedrossian and N. Masmoudi, {\it Inviscid damping and the asymptotic stability of planar shear flows in the 2D Euler equations},
 Publ. Math. Inst. Hautes \'{E}tudes Sci, 122(2015), 195-300.

\bibitem{BM2} J. Bedrossian, N. Masmoudi and V. Vicol, {\it Enhanced dissipation and inviscid damping in the inviscid limit of the Navier-Stokes equations near the two dimensional Couette flow}, Arch. Ration. Mech. Anal., 219(2016), 1087-1159.

\bibitem{BM3} J. Bedrossian, P. Germain and N. Masmoudi,
{\it On the stability threshold for the 3D Couette flow in Sobolev regularity}, Annals of Math., online.

\bibitem{Da} E. B. Davies, {\it Non-self-adjoint differential operators},
Bulletin of the London Mathematical Society, 34(2002), 513-532.

\bibitem{DSZ} N. Dencker, J. Sj\"{o}strand and M. Zworski,
{\it Pseudospectra of semiclassical (pseudo-) differential operators}, Comm. Pure Appl. Math., 57(2004), 384-415.

\bibitem{DW1}  W. Deng, {\it  Pseudospectrum for Oseen vortices operators},
 Int. Math. Res. Not.,  IMRN 2013, 1935-1999.

\bibitem{DW2}  W. Deng, {\it Resolvent estimates for a two-dimensional non-self-adjoint operator},
Comm. Pure Appl. Anal., 12 (2013), 547-596.

\bibitem{DR} P. G. Drazin and W. H. Reid, {\it Hydrodynamic stability},
Cambridge University Press, Cambridge, 1981.


\bibitem{Ga}  T. Gallay,
{\it Stability and interaction of vortices in two-dimensional viscous flows},
Discrete Contin. Dyn. Syst. Ser. S, 5(2012),  1091-1131.

\bibitem{GM} T. Gallay and Y. Maekawa, {\it Existence and stability of viscous vortices}, arxiv:1610.08384.

\bibitem{GW1} T. Gallay and C. E. Wayne,
{\it Invariant manifolds and the long-time asymptotics of the Navier-Stokes and vorticity equations on $\R^2$},
 Arch. Ration. Mech. Anal., 163 (2002), 209-258.

\bibitem{GW2} T. Gallay and C. E. Wayne,
 {\it Global stability of vortex solutions of the two-dimensional Navier-Stokes equation},
 Comm. Math. Phys., 255 (2005), 97-129.

\bibitem{GGN} I. Gallagher, Th. Gallay and F. Nier,
{\it Spectral asymptotics for large skew-symmetric perturbations of the harmonic oscillator}, Int. Math. Res. Notices, IMRN 2009, 2147-2199.

\bibitem{Kato} T. Kato, {\it Perturbation theory for linear Operators},
Grundlehren der mathematischen Wissenschaften 132, Springer, New York, 1966.

\bibitem{Ler} N. Lerner,
{\it Metrics on the phase space and non-selfadjoint pseudo-differential operators},
Pseudo-Differential Operators. Theory and Applications 3. Basel, Birkh\"{a}user, 2010.

\bibitem{LZ} Z. Lin and C. Zeng, {\it Inviscid dynamic structures near Couette flow}, Arch. Ration. Mech. Anal., 200 (2011), 1075-1097.

\bibitem{Mae} Y. Maekawa,
{\it Spectral properties of the linearization at the Burgers vortex in the high rotation limit},
J. Math. Fluid Mech., 13 (2011), 515-532.

\bibitem{PP1} A. Prochazka and D. I. Pullin, {\it On the two-dimensional stability of the axisymetric Burgers vortex},
 Phys. Fluids.,  7(1995), 1788-1790

\bibitem{PP2} A. Prochazka and D. I. Pullin, {\it Structure and stability of non-symmetric Burgers vortices},
 J. Fluid Mech., 163(1998), 199-228.

\bibitem{Tre} L. N. Trefethen, {\it Pseudospectra of linear operators}, SIAM Review, 39(1997), 383-406.

\bibitem{TE} L. N. Trefethen and M. Embree, {\it Spectra and Pseudospectra: the Behavior of nonnormal matrices and operators},
 Princeton University Press, New Jersey, 2005.

\bibitem{TRD} L. N. Trefethen, A. E. Trefethe, S. C. Reddy and T. A. Driscoll,
{\it  Hydrodynamic stability without eigenvalues}, Science, 261(5121), 578-584 (1993).

\bibitem{Vill2} C. Villani, {\it Hypocoercive diffusion operators}, International Congress of Mathematicians, Vol. III, 473-498,
Eur. Math. Soc, Z\"{u}rich, 2006.

\bibitem{Vill} C. Villani, {\it Hypocoercivity}, Mem. Amer. Math. Soc.,  202 (2009), no. 950.

\bibitem{WZZ1} D. Wei, Z. Zhang and W. Zhao, {\it Linear inviscid damping for a class of momotone shear flow in Sobolev spaces},
Comm. Pure Appl. Math., online.


\bibitem{WZZ2} D. Wei, Z. Zhang and W. Zhao, {\it Linear inviscid damping and enhanced dissipation for the Komogorov flow}, preprint.



\end{thebibliography}
\end{document}